\documentclass[11pt]{amsart}
\usepackage{amsfonts}
\usepackage{amscd}
\usepackage{amssymb}
\usepackage{graphics}
\usepackage{graphicx}

\usepackage{amsmath}

\usepackage{hyperref}
\hypersetup{colorlinks,citecolor=blue}

\newcommand{\BN}{{\mathbb{N}}}
\newcommand{\BR}{{\mathbb{R}}}

\newcommand{\gD}{\Delta}
\newcommand{\gd}{\delta}
\newcommand{\gb}{\beta}

\newcommand{\gC}{\Gamma}
\newcommand{\gc}{\gamma}
\newcommand{\gs}{\sigma}
\newcommand{\gS}{\Sigma}
\newcommand{\gO}{\Omega}

\newcommand{\gep}{\epsilon}

\newcommand{\gL}{\Lambda}
\newcommand{\ga}{\alpha}
\newcommand{\gt}{\tau}

\newcommand{\ti}[1]{\tilde{#1}}

\newcommand{\rank}{\text{rank}}

\newcommand{\SL}{\text{SL}}

\def\sub{\text{Sub}}

\newtheorem{prop}{Proposition}[section]
\newtheorem{thm}[prop]{Theorem}
\newtheorem{lem}[prop]{Lemma}
\newtheorem{cor}[prop]{Corollary}

\theoremstyle{definition}

\newtheorem{defn}[prop]{Definition}
\newtheorem{rem}[prop]{Remark}

\newtheorem{exam}[prop]{Example}

\catcode`\@=11
\long\def\@savemarbox#1#2{\global\setbox#1\vtop{\hsize\marginparwidth 
  \@parboxrestore\tiny\raggedright #2}}
\catcode`\@=12

\begin{document}
\author{Mikolaj Fraczyk and Tsachik Gelander}

%\address{Department of Mathematics, University of Chicago\\
%5734 S University Ave, Chicago, IL 60637\\
%\email{mfraczyk@math.uchicago.edu}

%\address{Mathematics and Computer Science\\
%Weizmann Institute\\
%Rechovot 76100, Israel\\}
%\email{tsachik.gelander@gmail.com}
%\thanks{Supported by ISF-Moked grant 2095/15}

\date{\today}

\title{Infinite Volume and Infinite Injectivity Radius}

\maketitle

\begin{abstract}
We prove the following conjecture of Margulis. Let $G$ be a higher rank simple Lie group and let $\gL\le G$ be a discrete subgroup of infinite covolume. Then, the locally symmetric space $\gL\backslash G/K$ admits injected balls of any radius. This can be considered as a geometric interpretation of the celebrated Margulis normal subgroup theorem. However, it applies to general discrete subgroups not necessarily associated to lattices.
Yet, the result is new even for subgroups of infinite index of lattices. 
We establish similar results for higher rank semisimple groups with Kazhdan's property (T). We prove a stiffness result for discrete stationary random subgroups in higher rank semisimple groups and a stationary variant of the Stuck--Zimmer theorem for higher rank semisimple groups with property (T). We also show that a stationary limit of a measure supported on discrete subgroups is almost surely discrete.
\end{abstract}

\section{Introduction}

\subsection{The statement for simple groups}
The main motivation for this paper was to prove the following result which confirms a conjecture of G.A. Margulis:

\begin{thm}\label{thm:simple}
Let $G$ be a connected centre-free simple Lie group of real rank at least $2$ and let $X=G/K$ be the associated symmetric space. Let $\gL\le G$ be a discrete group of infinite covolume.
Then for every $r>0$ there is a point $p\in \gL\backslash X$ where the injectivity radius is at least $r$.
\end{thm}

Theorem \ref{thm:simple} can be regarded as a geometric interpretation of Margulis' normal subgroup theorem for lattices in $G$. In particular the analogous result is false if $G$ has rank one.

\begin{exam}
Let $G$ be a simple group of rank one and let $\gC\le G$ be a uniform lattice. Then $\gC$ is Gromov hyperbolic and hence admits a nontrivial normal subgroup $\gD$ of infinite index. The infinite volume locally symmetric space $M=\gD\backslash G/K$ has bounded injectivity radius. Indeed, let $\ga\in\gD\setminus\{1\}$, let $\gO$ be a compact fundamental domain for $\gC$ in $G/K$ and let $D=\max \{ d(x,\ga\cdot x): x\in \gO\}$. Then $\text{InjRad}_M(p)\le D/2$ at any point $p\in M$. To see this consider a lift $\tilde p\in G/K$ and let $\gc\in \gC$ be an element such that $\gc^{-1}\cdot \tilde p\in\gO$. Then the element $\gc\ga\gc^{-1}$ is in $\gD$ and has displacement at most $D$ at $\tilde p$.
\end{exam}

We prove Theorem \ref{thm:simple} by showing that for a certain bi-$K$-invariant probability measure $\mu=\mu_G$ on $G$ (see \S \ref{sec:random-discrete}) the random walk on $\gL\backslash X$ eventually spends most of the time in the $r$-thick part for every $r$.

\begin{thm}\label{thm:simple-RW} (Corollary \ref{cor:RW-on-M})
Let $M=\gL\backslash X$ be an $X$-orbifold of infinite volume. Let $x_0\in X$ be an arbitrary point, set 
$\nu_n =\frac{1}{n}\sum_{i=0}^{n-1}\mu^{(i)}*\gd_{x_0}$ and let $\overline{\nu_n}$ be the pushforward of $\nu_n$ to $M$ via the covering map.
For every $r>0$ and $\gep>0$ there is $N$ such that $\overline{\nu_n}(M_{\ge r})\ge 1-\gep$ for every $n\ge N$,
%$$
% \frac{1}{n}\sum_{i=0}^{n-1}\mu^{(i)}(\{g\in G: gx_0\in M_{\ge r}\})\ge 1-\gep,
%$$
where $M_{\ge r}$ denotes the $r$-thick part of $M$ and $\mu^{(i)}$ the $i$'th convolution of $\mu$.
\end{thm}

\begin{rem}
$(i)$ It is possible to deduce the celebrated normal subgroup theorem of Margulis by comparing Theorem \ref{thm:simple-RW} for $\gL\backslash G$ with the result of Eskin and Margulis \cite{EM} about random walks on $\gC\backslash G$ where $\gC\le G$ is a lattice and $\gL\lhd\gC$ is a normal subgroup.
However, our proof relies on the Stuck--Zimmer theorem (or rather a variant of this theorem for stationary measures, see Theorem \ref{thm:stationary-Stuck--Zimmer}) which relies on the intermediate factor theorem of Nevo and Zimmer \cite{NZ-IFT}. Thus, as in Margulis' original proof of the classical normal subgroup theorem, this approach also traces back to the factor theorem of Margulis \cite{Ma-FT}.

%is a generalization of the NST proved using methods based on the original proof of the NST. Hence it should not be considered as an alternative approach for the classical normal subgroup theorem of Margulis.

$(ii)$ It follows from Theorem \ref{thm:simple} that for higher rank manifolds, finite volume is equivalent to bounded injectivity radius. It might be interesting to obtain a quantitative version of that statement. 
It is also possible to deduce the result of the seven authors \cite[Theorem 1.5]{7s} from Theorem \ref{thm:simple-RW} by a straightforward compactness argument.
%We remark also that \cite[Theorem 1.5]{7s} follows from Theorem \ref{thm:simple-RW} by a compactness argument.
\end{rem}

\subsection{Confined subgroups of semisimple groups}
More generally, let $G$ be a connected centre-free semisimple Lie group. We shall say that a discrete subgroup $\gL\le G$ is {\it confined} if there is a compact subset $C\subset G$ such that $\gL^g\cap C\setminus\{1\}\ne\emptyset$ for every $g\in G$. In other words, $\gL$ is confined if and only if the locally symmetric space $\gL\backslash G/K$ has bounded injectivity radius. Theorem \ref{thm:simple} is a special case of the following result (see Theorem \ref{thm:u.s.-alternative} for a more general statement where rank one factors with Kazhdan's property (T) are allowed):

\begin{thm}\label{thm:ss-us}
Suppose that all the simple factors of $G$ are of real rank at least $2$. A discrete subgroup $\gL\le G$ is confined if and only if there is a nontrivial normal subgroup $H\lhd G$ such that $\gL\cap H$ is a lattice in $H$.
\end{thm}

We shall say that a subgroup $\gD\le\gC$ is a {\it conjugate limit} of $\gL$ if $\gD$ belongs to the closure of the conjugacy class $\overline{\gL^G}\subset\sub(G)$ in the Chabauty topology (see \S \ref{sec:random-discrete}). It is obvious that a conjugate limit of a confined group is also confined.

\subsection{Random walk on $G/\gL$}
Let $G$ be a connected semisimple Lie group.
Let $\gL\subset G$ be a discrete subgroup of $G$. 
Consider the sequence of probability measures 
\[
 \nu_n:=\frac{1}{n}\sum_{i=0}^{n-1}\int_G \delta_{\gL^g}d\mu^{(n)}(g).
\]
Every weak-* limit of the sequence $\nu_n$ is a $\mu$-stationary measure supported on the closure of the orbit of $\gL$ in $\sub(G)$. We would like to know if stationary random subgroups constructed in this way retain some properties of $\gL$. An essential result of this paper is that every stationary limit of $\nu_n$ is almost surely discrete (see also Theorem \ref{thm:dSRS} for a general statement). 
 
\begin{thm}\label{thm:intro-RW}
Let $\nu_\infty$ be a weak-* limit of $(\nu_n)$. Then $\nu_\infty(\sub_d(G))=1$ where $\sub_d(G)$ is the Chabauty open set of discrete subgroups of $G$.
\end{thm}

\subsection{Stiffness and the Stuck--Zimmer theorem for stationary measures}
In view of Theorem \ref{thm:intro-RW} we are led to study stationary measures supported on the space $\sub_d(G)$ of discrete subgroups of $G$, that is, $\mu$-stationary measures $\nu$ with $\nu(\sub_d(G))=1$. Relying on the remarkable theorems of Nevo and Zimmer \cite{NZ,NZ99} we establish that every discrete stationary random subgroup of a higher rank group is (under a certain irreducibility assumption with respect to the rank one factors) an invariant random subgroup. 

%The following result can be read as ``higher rank stationary discrete subgroups are invariant":

\begin{thm}\label{thm:stiffness}
Let $G$ be a connected centre-free semisimple Lie group without compact factors and real rank at least two. Let $\nu$ be a $\mu$-stationary measure on $\sub_d(G)$. Suppose that $\nu$-almost every random subgroup intersects trivially every rank one factor of $G$.
%or the projection to the factor is non-discrete. 
Then $\nu$ is invariant. 
\end{thm}

Theorem \ref{thm:stiffness} is a consequence of a decomposition theorem for stationary measures (Theorem \ref{thm:stationary-decomposition}). Observe that in view of Theorem \ref{thm:stationary-decomposition}, the condition that the intersection with every rank one factor is trivial implies that the projection to every such factor is either non-discrete or trivial. The analog of Theorem \ref{thm:stiffness} does not hold for rank one groups (see Example \ref{exam:rank-1}), yet we prove a weak variant of the Nevo--Zimmer factor theorem for rank one groups (see Theorem \ref{prop:rank-one}).

\begin{rem}
Theorem \ref{thm:stiffness} specialized to irreducible lattices in $G$ can also be deduced from the 
recent results \cite{BH,Creutz}.
\end{rem}

For semisimple Lie groups with Kazhdan's property (T) we deduce the following generalization of the Stuck--Zimmer theorem for discrete stationary random subgroups:

\begin{thm}\label{thm:stationary-Stuck--Zimmer}
Let $G$ be a connected centre-free semisimple Lie group without compact factors. Suppose that $G$ has real rank at least $2$ and Kazhdan's property (T). Let $\nu$ be an ergodic $\mu$-stationary measure on $\sub_d(G)$. Suppose that $\nu$-almost every random subgroup intersects trivially every rank one factor of $G$. Then there is a semisimple factor $H\lhd G$ and a lattice $\gC\le H$ such that $\nu=\nu_\gC$. 
\end{thm}

Here $\nu_\gC$ denotes the invariant random subgroup obtained by the pushforward of the probability measure from $H/\gC$ to $\sub(H)\subset\sub(G)$ via the map $h\gC\mapsto h\gC h^{-1}$.

\subsection{The conclusion}

Combining Theorem \ref{thm:intro-RW}, Theorem \ref{thm:stationary-Stuck--Zimmer} and local rigidity we deduce the following:
% (see \S \ref{sec:US} below for other results of this nature):

%When all the factors of $G$ are higher rank we obtain the following (see also Theorem \ref{thm:no-discrete-projections} for a more general statement that allows rank one rank one factors but assumes property (T)):

\begin{thm}
Let $G$ be a connected centre-free semisimple Lie group without compact factors. Suppose that $G$ has real rank at least $2$ and Kazhdan's property (T).
Let $\gL\le G$ be a discrete subgroup. Suppose that for every nontrivial semisimple factor $H\lhd G$ the intersection $\gL\cap H$ is not a lattice in $H$. Suppose also that no discrete conjugate limit of $\gL$ intersects a rank one factor of $G$ in a Zariski dense subgroup. Then $\frac{1}{n}\sum_{i=0}^{n-1}\mu_G^{(i)}*\gd_\gL$ weakly converges to $\gd_{\{1\}}$.
\end{thm}

See \S \ref{sec:US} below for more general results of this nature. In particular Theorem \ref{thm:g.s.s.-R-W} deals with a general discrete subgroup $\gL$ of a general semisimple group $G$, not even assuming property (T). Loosely speaking, if $\gL$ does not contain a lattice of a higher rank semisimple factor then every stationary limit is supported on discrete subgroups of the product of rank one factors of $G$.

\begin{rem}
The main results of this work holds also for analytic groups over non-archimedean local fields. The same proofs can be carried out in that generality with minor adaptations. However since the Nevo--Zimmer factor theorem \cite{NZ,NZ99} is written for real Lie groups, we decided to restrict to that case as well. We remark that the proof of the Nevo--Zimmer theorem also applies with minor changes to the non-archimedean setup. 
\end{rem}

%\medskip

{\bf Acknowledgment.}
We thank Uri Bader for sharing with us some insights concerning stationary measures and Poisson boundaries. Our work was supported by the National Science Foundation under Grant No.\ DMS-1928930 while the authors participated in a program hosted by the Mathematical Sciences Research Institute in Berkeley, California, during the Fall 2020 semester. The second author was partially supported by the Israel Science Foundation grant No.\ 2919/19.

%%%%%%%%%%%%%
%%%%%%%%%%%%%
%%%%%%%%%%%%%

%\medskip
%
%For a locally symmetric space $M=\gC\backslash X$ demote by
%$$
% R(M)=\sup\{\text{InjRad}_M(x):x\in M\}\in \BR\cup\{\infty\}
%$$ 
%the supermum injectivity radius of $M$. We can reformulate Theorem \ref{thm:main} as the equivalence of the finiteness of the supermum injectivity radius and the volume: 
%
%\begin{thm}
%Let $G,\Gamma$ and $M$ be as above, then
%$R(M)<\infty$ iff $\vol(M)<\infty$.
%\end{thm}
%centre
%In view of \cite[Theorem ?]{7} we obtain the following:
%
%\begin{thm}
%For every $r>0$ there is $v>0$ such that $\vol(M)>v$ implies $R(M)>r$.
%\end{thm}
%
%In other words for a sequence of $X$-manifolds $M_n$, we have $R(M_n)\to\infty$ iff $\vol(M_n)\to\infty$.
%
%Consider now the case where $G$ is semisimple. We prove the following alternative:
%
%\begin{thm}
%Let $G$ be a connected centre-free semisimple Lie group without compact factors and with Kazhdan's property (T). Let $\Gamma\le G$ be a discrete group which projects densely to every rank one factor of $G$. Then either 
%\begin{itemize}
% \item $\Gamma\backslash G/K$ admits injected balls of arbitrarily large radius, or
% \item there is a nontrivial (semisimple) factor $H\lhd G$ such that $\Gamma\cap H$ is a lattice in $H$.
%\end{itemize}
%\end{thm}
%
%

%%%%%%%%%%%
%%%%%%%%%%%
%%%%%%%%%%%

\section{Random walks on the space of discrete subgroups}\label{sec:random-discrete}

Let $G$ be a connected centre-free semisimple Lie group without compact factors and let $K$ be a maximal compact subgroup. 

\subsection{The measure associated to $G$}\label{sec:nu_G}
We let $\mu_G$ be the probability measure 
$$
 \mu_G=\eta_K*\gd_s*\eta_K
$$ 
defined in \cite[\S 8]{GLM}. Here $\eta_K$ denotes the normalized Haar measure on $K$ and $s\in G$ is a certain regular semisimple element with sufficiently good expanding properties when acting on the unipotent radical of a fixed minimal parabolic subgroup (see \cite[\S 6]{GLM}). We will not make any use below of the explicit properties of the element $s$. All that we need to know is that $\mu_G$ is bi-$K$-invariant and that Theorem \ref{thm:eq-Marg} holds.
We will refer to $\mu_G$ as the probability measure associated to $G$. Note that if $G$ is semisimple, $G=G_1\times\ldots\times G_n$, we have $\mu_G=\mu_{G_1}*\cdots*\mu_{G_n}$
and the measures $\mu_{G_i}$ pairwise commute.
%When there are no confusion about $G$ we will omit the subscript. 

\subsection {The discreteness radius}
%Consider the symmetric space $X=G/K$ associated with $G$. Let $x_0\in X$ be the fixed point of $K$. Let $d_X$ be the Riemannian metric on $X$.
%We endow $G$ with a metric 
%\[
% d(g,h)=\max \{d_X(h^{-1}g\cdot y,y):~y\in B_X(x_0,1)\}.
%\]
%The ball $B_G(r):=B_G(r,1)$ consists of those elements $g$ whose displacement evaluated at the points of the unit ball around $x_0$ in $X$ is 
%at most $r$. 

Fix a norm $\|\cdot\|$ on the Lie algebra $\text{Lie}(G)$ such that $\text{exp}:\text{Lie}(G)\to G$ restricted to the unit ball $B(1)=\{X\in\text{Lie}(G):\| X\|<1\}$ is a well defined diffeomorphism. For $r\le 1$ denote $B(r)=\{X\in\text{Lie}(G):\| X\|<r\}$.
For a discrete group $\gL\subset G$ set
$$
 \mathcal{I}(\gL)=\sup\{r\le 1: \exp B(r)\cap \Gamma=\{1\}\}.
$$ 
We call $\mathcal{I}(\Gamma)$ the discreteness radius of $\gL$.

%Note that $\mathcal{I}(\Gamma)/2$
%equals to the maximal injectivity radius centred somewhere in the closed unit ball around the projection $\pi(x_0)$ of $x_0$ to $\Gamma\bs X$. It follows that 
%$$
% \text{InjRad}_M\pi(x_0)\le\mathcal{I}(\Gamma)/2\le \text{InjRad}_M\pi(x_0)+1.
%$$

\subsection{The Margulis function on the space of discrete subgroups of $G$}

We denote by $\sub(G)$ the space of closed subgroups of $G$ equipped with the Chabauty topology and by $\sub_d(G)$ the subset of discrete subgroups of $G$. Since $G$ has no small subgroups, $\sub_d(G)$ is open (see \cite[Lemma 1.1]{KM-IRS}).
We will say that a measure $\nu$ on $\sub(G)$ is supported on the set of discrete subgroups if $\nu(\sub_d(G))=1$. 
A stationary measure supported on $\sub_d(G)$ will be called a {\it discrete stationary random subgroup}.

An essential result established in \cite{GLM} is that there is a positive constant $\gd=\gd(G)$ such that 
$$
 u(\gC):=\mathcal{I}(\gC)^{-\gd}
$$
satisfies Inequality (\ref{eq-Marg}) below, that is, it is a Margulis function on $\sub_d(G)$ with respect to $\mu_G$. 

\begin{thm}[\cite{GLM}, Theorem 1.5]\label{thm:eq-Marg}
There exist $0<c<1, b\geq 0$ such that, for every discrete subgroup $\gC\le G$,
\begin{equation}\label{eq-Marg}
 \int_G u(\gC^g) d\mu_G(g)\leq c u(\gC) +b. 
\end{equation} 
\end{thm}

We remark that the constants $\gd, c$ and $b$ are constructed explicitly in \cite{GLM} in order to prove certain effective results and in particular a quantitative version of the Kazhdan--Margulis theorem.  

\subsection{Stationary limit are discrete}

The main result of this section is that any stationary limit of a measure supported on discrete subgroups of $G$ is almost surely discrete. 
This is a key ingredient in the proofs of Theorem \ref{thm:simple} and Theorem \ref{thm:simple-RW}, as well as the results of \S \ref{sec:US}. 
%We will prove a more general result that concerns with all measures supported on the set of discrete subgroups. 

\begin{thm}\label{thm:dSRS}
Let $\nu$ be a probability measure on $\sub_d(G)$. Let 
$$
 \nu_n=\frac{1}{n}\sum_{i=0}^{n-1}\mu_G^{(n)}*\nu
$$ 
and let $\nu_\infty$ be a weak-* limit of $\nu_n$. Then $\nu_\infty$ is supported on the set of discrete subgroups of $G$, that is, $\nu_\infty$ is a discrete stationary random subgroup. 

\end{thm}

The proof is reminiscent of \cite{EM}. 
%Using the \emph{Margulis function} $u$ we will show that typical trajectory of a random walks spends most of the time outside of the thin part of $G/\Gamma$. This will quickly imply Theorem \ref{thm:dSRS}. 

\begin{proof} 
By restricting to compact subsets of $\sub_d(G)$, we may allow ourself to suppose that $\nu$ is compactly supported and that $A:=\int u(\Lambda)d\nu(\Lambda)$ is finite. 

To prove Theorem \ref{thm:dSRS} we need to show that 
%\begin{equation}\label{eq-Discrete1}
%  \lim_{\gep\to 0}\int_{G} 1_{\{\mathcal{I}(\Lambda)<\gep\}}(\Lambda)d\nu_\infty(\Lambda)= 0.
%\end{equation}
\begin{equation}\label{eq-Discrete1}
  \lim_{\gep\to 0}\nu_\infty({\{\Lambda:\mathcal{I}(\Lambda)<\gep\}})= 0.
\end{equation}
%For the proof of the Theorem we may replace $\mu$ be a finite self-convolution. Indeed any weak limit for $\mu$ is a convex combination of weak limits for $\mu^n$. Because of that, it wont hurt to assume that the condition (\ref{eq-Marg}) holds for $n_0=1$. 
%Let $X_n$ be the $n$-step of the random walk on $G$ induced by $\mu$. 
Inequality (\ref{eq-Marg}) implies that 
$$
 \int u(\Lambda)d(\mu_G*\nu)(\Lambda)=\int u(\Lambda^g)d\mu_G(g)d\nu(\Lambda)\le cA+b.
$$
Furthermore, iterating Condition (\ref{eq-Marg}) and summing the resulting geometric series we get
\[ 
 %\int u(\Lambda)d\nu_n(\Lambda)=
 \int u(\Lambda)d(\mu_G^{(n)}*\nu)(\Lambda)< c^n A+C,
\] 
with $C:=b/(1-c)$, uniformly for all $n$. Set $M=A+C$.
Then, for every $\gt>0$ and $n\geq 1$ we have
%\[ 
% \int_G 1_{\{u(\gC^g)\geq M\gt^{-1}\}}d\mu^n(g)< \gt.
%\]
\[ 
 \nu_n(\{\Lambda:u(\Lambda)\geq M\gt^{-1}\})< \gt.
\]
%The function $u(\Lambda)$ tends to $\infty$ as $\mathcal{I}(\Lambda)\to 0$, so there exists $\gep>0$ such that $\mathcal{I}(\Lambda)\leq \gep$ implies $u(\Lambda)>M\gt^{-1}$. We get 
Setting $\gep=(\gt/M)^\frac{1}{\gd}$ we get
\[ 
 \nu_n(\{\Lambda:\mathcal{I}(\Lambda)\le\gep\})<\gt.
\]
Taking $n\to\infty$ gives 
$\nu_\infty(\{\Lambda:\mathcal{I}(\Lambda)\le\gep\})<\gt$,
and letting $\gt \to 0$ we get (\ref{eq-Discrete1}).
\end{proof}

%%%%%%%%%%%%%
%%%%%%%%%%%%%
%%%%%%%%%%%%%

\section{Essential results about discrete stationary random subgroups}

In this section we assemble some results about discrete stationary random subgroups that will be essential in the proofs of the main results.
We start by recalling the following classical result of Furstenberg (see also \cite[Theorem 2.16]{Bader-Shalom}):

\begin{thm}[Furstenberg]
Let $\mu$ be a probability measure on $G$ and let $(B,\nu_B)$ be the associated Poisson boundary. Let $X$ be a compact $G$-space and let $\nu$ be a probability measure on $X$. Then $\nu$ is $\mu$-stationary if and only if it is the $\nu_B$-barycentre of some measurable map $\xi:B\to \text{Prob}(X)$. 
\end{thm}

Recall that $\text{Prob}(X)$ is the weak-$*$ compact space of probability measures on $X$ and the $\nu_B$-barycentre of $\xi$ is the measure $\int \xi(\omega)d\nu_B(\omega)$. Thus, to be $\mu$-stationary is a property of the Poisson measure $\nu_B$ rather than the specific choice of $\mu$. In particular:

\begin{cor}\label{cor:same-stationary}
Suppose that $\mu_1$ and $\mu_2$ are two probability measures on $G$ corresponding to the same Poisson measure $(B,\nu_B)$. Then $(X,\nu)$ is $\mu_1$-stationary if and only if it is $\mu_2$-stationary.
\end{cor} 

This corollary allows us to vary the measure $\mu$ according to our needs, when analysing stationary measures, as long as we do not change the Poisson measure.

Let us now restrict to the case where $G$ is a connected semisimple Lie group without compact factors and $K$ is a maximal compact subgroup.
In view of the Iwasawa decomposition, $K$ acts transitively on $G/P$ where $P\le G$ is a minimal parabolic subgroup. Therefore there is a unique Borel regular $K$-invariant probability measure $\nu_P$ on $G/P$. Let $\mu$ be a probability measure on $G$ with support generating a Zariski dense subgroup. The space $G/P$ supports a unique $\mu$-stationary probability measure, which makes it the Poisson boundary for $(G,\mu)$  (see for example \cite{Fur63}).  Whenever $\mu$ is a bi-$K$-invariant probability measure on $G$, the unique $\mu$-stationary measure on $G/P$ must be $K$-invariant, hence equal to $\nu_P$. We record that the measure $\mu_G$ introduced in \S \ref{sec:random-discrete} is bi-$K$-invariant, so $(G/P,\nu_P)$ is the Poisson boundary for $\mu_G$.

The measure $\mu_G$ is not smooth\footnote{It is possible to show that some finite power $\mu_G^{(n)}$ is smooth, but in view of Corollary \ref{cor:same-stationary} we will not need that.} but obviously $G$ admits some smooth bi-$K$-invariant probability measure. Therefore, in view of Corollary \ref{cor:same-stationary}, statements about stationary measures with respect to a measure on $G$ which is assumed to be smooth will remain valid also for the non-smooth measure $\mu_G$.
For simplicity we will consider, without repeating it, only bi-$K$-invariant probability measures $\mu$ on $G$ although most of the statements apply to general (smooth) measures. This will allow us to make use of the special properties of $\mu_G$ and in particular properties that follow from Inequality (\ref{eq-Marg}) and from Theorem \ref{thm:dSRS}.
Thus in the sequel $\mu$ will refer to an arbitrary bi-$K$-invariant probability measure on $G$ while $\mu_G$ is the specific measure given in \S \ref{sec:nu_G}. 

Let us denote by $\text{d-SRS}(G)$ the space of all discrete $\mu$-stationary random subgroups of $G$. 
The following result from \cite{GLM} is a consequence of Inequality (\ref{eq-Marg}):

\begin{thm}[\cite{GLM}, Theorem 1.2]\label{lem:WUD}
The space $\text{d-SRS}(G)$ is weakly uniformly discrete. That is, for every $\gep>0$ there is an identity neighbourhood $U\subset G$ such that for every $\nu\in \text{d-SRS}(G)$, 
$$
\nu(\Lambda:\Lambda\cap U\ne\{1\})\le\gep.
$$
\end{thm}
%For the convenience of the reader we include the short proof:
%\begin{proof}

Let us recall the straightforward proof when $\nu$ is a compactly supported discrete stationary random subgroup. In that case Inequality (\ref{eq-Marg}) gives:
$$
 \int u(\Lambda)d\nu=\int u(\Lambda)d\mu*\nu\le c\int u(\Lambda)d\nu+b.
$$
It follows that $\int u(\Lambda)d\nu\le C:=b/(1-c)$. Thus $\nu(\{\Lambda:u(\Lambda)\ge C/\gep\})\le \gep$.
Thus we can take $U$ to be the ball of radius $(\gep/C)^{\frac{1}{\gd}}$ around $1_G$.
%\end{proof}

%This implies:

\begin{cor}[\cite{GLM}, Corollary 1.6]
The space $\text{d-SRS}(G)$ is weak-* compact.
\end{cor}

\begin{proof}
It is obvious that a limit of stationary measures is stationary and it follows from Theorem \ref{lem:WUD} that a limit of discrete stationary random subgroups is also discrete.
\end{proof}

%We deduce from \cite[\S 2]{Bader-Shalom} that 
The extreme points of the compact convex space $\text{d-SRS}(G)$ are ergodic (the converse is also true, see \cite[Corollary 2.7]{Bader-Shalom}). Thus by the Choquet integral theorem every discrete stationary random subgroup is a barycentre of some probability measure on the set of ergodic discrete stationary random subgroups. This fact allows us to assume ergodicity when proving various results about discrete stationary random subgroups.

Let $P$ be a minimal parabolic subgroup of $G$ with Langlands decomposition $P=MAN$.

\begin{lem}\label{lem-StatioHS}
Let $H\subset G$ be an algebraic subgroup. The homogeneous space $G/H$ admits a $\mu$-stationary probability measure if and only if $H$ contains $AN$, up to conjugacy. 
\end{lem}
After writing the proof of Lemma \ref{lem-StatioHS} we found out that the same statement was already proven in \cite[Prop 3.2]{NZ}. Our proof is different and self-contained so we present it below for completeness.
\begin{proof}
If $H$ contains a conjugate of $AN$ then $G/H$ is compact so the existence of a $\mu$-stationary probability measure is clear. Now assume that $\nu$ is a $\mu$-stationary probability measure on $G/H$. Since $(G/P,\nu_P)$ is the Poisson boundary we get a $G$-equivariant map $\kappa\colon G/P\to \text{Prob}(G/H)$ such that $\nu=\int_{G/P}\kappa(gP)d\nu_P(gP).$ The measures $\kappa(gP)$ are $gPg^{-1}$-invariant probability measures almost surely. By conjugating $P$ if necessary we can assume without loss of generality that $\kappa(P)$ is a $P$-invariant probability measure. 
We will show that the existence of such a measure implies that $H$ contains a conjugate of $AN$. Let $\nu_0$ be an ergodic $AN$-invariant component of $\nu$. By Chevalley's theorem \cite[5.1]{Bor} there is a rational representation $G\curvearrowright V$ and a line $[v]\in \mathbb P^1(V)$ such that ${\rm Stab}_G [v]=H$. We will think of $\nu_0$ as an ergodic $AN$-invariant probability measure on $\mathbb P^1(V)$ supported on the $G$ orbit of $[v]$. By Lemma \ref{lem-ProjActions}, we must have $\nu_0=\delta_{[gv]}$ for some $g\in G$ and $[gv]$ must be fixed by $AN$. If follows that $gHg^{-1}\supset AN$.
%Since $AN$ is a solvable group, 
\end{proof}

The following Lemma \ref{lem-ProjActions} is required also in the proof of 
Lemma \ref{lem-Projective}.

\begin{lem}\label{lem-ProjActions}
Let $V$ be a rational real representation of $AN$. Then, any ergodic $AN$-invariant probability measure on $\mathbb P^1(V)$ is supported on a single line fixed by $AN$. 
\end{lem}
\begin{proof}
Let $\nu$ be an ergodic $AN$-invariant probability measure on $\mathbb P^1(V)$. The group $N$ is a unipotent algebraic group so any rational representation of $N$ has a fixed vector \cite[4.8]{Bor}. Let $V'$ be the subspace of $N$-fixed vectors. It is preserved by $A$. Since $A$ is an $\mathbb R$-split torus, any rational representation of $A$ over $\mathbb R$ decomposes into a direct sum of $A$-eigenspaces \cite[8.4]{Bor}. We deduce that there exists a one-dimensional subspace $V_1\subset V'$ which is preserved by $AN$. Reasoning inductively we construct a basis $e_1,\ldots ,e_d$ of $V$ such that $V_i=\mathbb R e_1+\ldots +\mathbb R e_i$ are preserved by $AN$ and each $e_i$ is an eigenvector of $A$. Write $\chi_i$ for the character of $A$ such that $ae_i=\chi_i(a)e_i.$ %We extend $\chi_i$ to characters of $AN$ by letting $N$ act trivially.

Let $i_0$ be the minimal index for which $\nu(\mathbb P^1(V_{i_0}))>0$. The sets $\mathbb P^1(V_{i_0})$ are all $AN$-invariant, so by ergodicity $\nu(\mathbb P^1(V_{i_0})\setminus \mathbb P^1(V_{i_0-1}))=1.$ Consider the map $\iota\colon \mathbb P^1(V_{i_0})\setminus \mathbb P^1(V_{i_0-1})\to V_{i_0}$
\[ \iota([x_1,\ldots,x_{i_0-1},x_{i_0}]):= (x_1/x_{i_0},\ldots, x_{i_0-1}/x_{i_0},1).
\]
Let $W$ be the space $V_{i_0}$ with the action of $AN$ given by $an w=\chi_{i_0}(a)^{-1}anw.$ The map $\mathbb P^1(V_{i_0})\setminus \mathbb P^1(V_{i_0-1})\to W$ is $AN$-equivariant. The measure $\iota_*\nu$ is an ergodic $AN$-invariant probability measure on $ W$. By Lemma \ref{lem-RatRepProb}, $\iota_*\nu=\delta_w$ for some $AN$ invariant vector $w\in W$. This means that $\nu$ itself was supported on the line $[w+e_{i_0}]$, which is fixed by $AN$. 
\end{proof}
\begin{lem}\label{lem-RatRepProb} Let $W$ be a finite dimensional rational representation of $AN$. Any ergodic $AN$-invariant probability measure on $W$ is of the form $\delta_w$ for an $AN$-fixed vector $w\in W$. 
\end{lem}
\begin{proof}
We prove the statement by induction, with cases $\dim W=1,2$ serving as the induction base.  If $\dim W=1$ then $N$ acts trivially because unipotent actions must fix a non-zero vector. If $A$ acts non-trivially then $\nu=\delta_{0}$ because no probability measure on $\mathbb R\setminus \{0\}$ can be invariant under dilations. Otherwise, $AN$ acts trivially and $\nu=\delta_w$ for some $w\in W$, by ergodicity. 

Suppose $\dim W=2$. Let $e_1,e_2$ be a basis of $W$ such that $e_1$ is $N$-invariant and both $e_1,e_2$ are eigenvectors of $A$. If the action of $N$ is nontrivial, then there exists $n_0\in N$ and $x_0\in \mathbb R\setminus \{0\}$ such that $n_0(e_2)=e_2+x_0 e_1$. For every compact set $K\subset W$ the intersection $\bigcap_{k\in \mathbb Z} n_0^k K\subset K\cap (\mathbb R e_1),$ so every $N$-invariant probability measure on $W$ must have $\nu(\mathbb Re_1)=1$ and the lemma follows now from the one dimensional case. Consider the case that the action of $N$ is trivial. There are rational characters $\chi_1,\chi_2$ of $A$ such that $a e_i=\chi_i(a)e_i$. Let $W'=\sum_{\chi_i=1} \mathbb Re_i$. Then for any compact subset $K\subset W$ we have $\bigcap_{a\in A} aK\subset K\cap W'$. It follows that $\nu$ is supported on $W'$. The action of $AN$ on $W'$ is trivial so by ergodicity $\nu=\delta_w$ for some $w\in W'$. 

We move to the general case $\dim W\geq 3$.  Let $W_1$ be a one dimensional subspace preserved by $AN$. By the inductive hypothesis, the pushforward of $\nu$ to $W/W_1$ is supported on a single element. It follows that $\nu$ is supported on a single line $w+W_1$ for some $w\in W$. Hence $\nu(W_1+\mathbb R w)=1$. We have $\dim (W_1+\mathbb R w)\leq 2$ so the lemma follows from the first two cases.
\end{proof}

\begin{lem}\label{lem-RootGen} Let $\Phi$ be a simple root system (see \cite[\S 4.2]{Knapp}) spanning a Euclidean space $V$. Let $V_0\subset V$ be a proper subspace and let $S=\Phi\cap V_0$. Then \begin{enumerate}
\item $\Phi\setminus S$ generates the whole root system $\Phi$.
\item $\Phi\setminus S$ spans $V$.
\end{enumerate}
\end{lem}
\begin{proof}
For any subset $F\subset \Phi$, let $\langle F\rangle$ be the smallest subset of $\Phi$ containing $F$ that is closed under taking reflections. For the first assertion we need to show $\langle \Phi\setminus S\rangle=\Phi$. Let $S_0=S, E_0=V_0$. We define inductively the subsets $S_i$ and subspaces $E_i$.
\[ E_{i+1}=E_i\cap (\Phi\setminus S_i)^{\perp}, \quad S_{i+1}=\Phi\cap E_{i+1}.\]
By construction $S_{i+1}\subset S_i$ and $E_{i+1}\subset E_i$. We argue that $\langle \Phi\setminus S_i\rangle =\langle \Phi\setminus S_{i+1}\rangle.$ The inclusion $\langle \Phi\setminus S_i\rangle\subset \langle \Phi\setminus S_{i+1}\rangle$ is clear, so it is enough to show that $S_i\setminus S_{i+1}\subset \langle \Phi\setminus S_i\rangle.$ Let $\lambda\in S_i\setminus S_{i+1}$. By definition, we must have $\lambda\not \in E_{i+1}$, so there exists a root $\alpha\not \in S_i$ such that $\langle \lambda,\alpha\rangle\neq 0$. The reflection (\cite[p. 69]{Knapp})
\[ s_{\alpha}(\lambda)=\lambda-\frac{2\langle\lambda,\alpha\rangle}{\|\alpha\|^2}\alpha\] is not in $S_i$ because $\lambda\in E_i$ and $\alpha\not\in E_i$. We deduce that $\lambda=s_\alpha(s_\alpha(\lambda))\in \langle \Phi\setminus S_i\rangle.$ This proves that $\langle \Phi\setminus S_i\rangle =\langle \Phi\setminus S_{i+1}\rangle.$ 

The sequence of sets $S_i$ eventually stabilizes, so we have $S_{i+1}=S_i$ for some $i$. Therefore $S_i\subset (\Phi\setminus S_i)^\perp.$ The root system $\Phi$ is simple, so this is possible only if $S_i=\emptyset.$ We deduce that $\langle \Phi\setminus S_0\rangle=\langle \Phi\setminus S_1\rangle=\ldots=\langle \Phi\setminus S_i\rangle=\Phi.$ The second assertion trivially follows from the first. 
\end{proof}

\begin{rem}
The assumption that $S$ is contained in a proper subspace is necessary in Lemma \ref{lem-Projective}. For example, the root systems $B_2,C_2,G_2,B_3,C_3$ decompose as unions of two proper root subsystems.
\end{rem}

\begin{lem}\label{lem-ANCentralizers}
The centralizer of $AN$ in $G$ is trivial. 
\end{lem}
\begin{proof}

The lemma easily reduces to the case where $G$ is simple, so let us assume that from now on. 
The centralizer of $A$ is $MA$. It follows that $C_G(AN)$ is a normal subgroup of the reductive group $MA$. Therefore it is enough to show that every semisimple element in  $C_G(AN)$ is trivial. Let $\gamma\in C_G(AN)$ be a semisimple element. Let $T$ be a maximal torus of $G$ containing $\gamma$. Let $\Phi$ be the set of roots of $T$ in $\frak g_\mathbb C$.  The roots naturally lie in the Euclidean space $X^*(T)\otimes \mathbb R$ and form a root system in  the classical sense. Let $E\subset X^*(T)\otimes \mathbb R$ be the proper subspace spanned by the roots in $\frak m_\mathbb C$. Then, $\Phi\setminus E$ is the set of roots in $\frak n_\mathbb C$ and their opposites. Since $\gamma$ commutes with $N$, the adjoint action must be trivial on $\frak n$ and therefore trivial on $\frak n_\mathbb C$. It follows that $\xi(\gamma)=1$ for every root $\xi\in \Phi\setminus E$. On the other hand, by Lemma \ref{lem-RootGen}, the roots in $\Phi\setminus E$ generate $\Phi$, so $\xi(\gamma)=1$ for every root in $\Phi$. This means that ${\rm Ad}(\gamma)$ acts trivially on $\frak g_\mathbb C$. Since $G$ is centre-free, $\gamma=1$. 
\end{proof}

Let $2^G$ denote the compact space of closed subsets of $G$ equipped with the Chabauty topology and consider the $G$ action by conjugation.
Let $\mathcal{F}(G)\subset 2^G$ be the set of finite subsets of $G$.

\begin{lem}\label{lem-StatioFin}
For any finite set $F\subset G\setminus\{1\}$, 
 the $\mu_G$-stationary measure 
 $$
 \lim_{n\to\infty}\frac{1}{n}\sum_{i=0}^{n-1}\mu_G^{(n)}*\gd_F
$$ 
 is the Dirac measure on the empty set.
\end{lem} 

\begin{proof}
It is enough to prove the assertion for $F=\{\gamma\}, \gamma\in G\setminus \{1\}$.  Moreover, by projecting to a factor we may suppose that $G$ is simple.
Let $\nu$ be the limit in question. 
If it is not the Dirac delta on the empty set then it gives us a $\mu$-stationary probability measure on the closure of the conjugacy class $\overline{\gamma^G}$. 
%All the elements in $\overline{\gamma^G}$ have the same eigenvalues as $\gamma$ so 
It follows from the Jordan--Chevalley decomposition together with the Dynkin--Kostant classification of nilpotent orbits (see \cite{CM}) that $\overline{\gamma^G}$ is a finite union of conjugacy classes. By restricting to one, say $\gc_0^G$, we obtain a finite positive $\mu$-stationary measure on $G/G_{\gamma_0}$.  By Lemma \ref{lem-StatioHS}, up to replacing $\gc_0$ by a conjugate, $G_{\gamma_0}$ contains $AN$. 
%By Lemma \ref{lem-ANCentralizers} 
By Lemma \ref{lem-ANCentralizers} we get $\gamma_0=1$. Thus, we are left with showing that $\nu$ cannot have an atom at $\{1\}$. Now if $\nu(\{1\})=\gep>0$ then $1$ is a conjugate limit of $\gc$ and $\gc$ is unipotent. In that case by applying Theorem \ref{thm:dSRS} to the discrete cyclic group $\langle \gc\rangle$  we can choose  an identity neighborhood $U$ such that $\lim_{n\to\infty}\frac{1}{n}\sum_{i=0}^{n-1}\mu_G^{(n)}*\gd_{\{\gc\}}(U)<\gep$, a contradiction.
\end{proof}

%Lemma \ref{lem-StatioFin} implies:
 
\begin{cor}\label{cor:st-fi}
The only ergodic $\mu$-stationary probability measures on $\mathcal{F}(G)$ are $\gd_{\{1\}}$ and $\gd_\emptyset$.
\end{cor}

\begin{proof}
In view of the dominated convergence theorem it follows from Lemma \ref{lem-StatioFin} that the only $\mu$-stationary probability measure on $G$ with respect to the action by conjugation  is $\gd_{\{1\}}$. Since the average of the Dirac measures on a finite set which is chosen randomly with respect to a $\mu$-stationary measure on $\mathcal{F}(G)$ produces such a measure, the corollary follows.
\end{proof}

We will also make use of the following beautiful result:

\begin{prop}[Bader--Shalom \cite{Bader-Shalom}]\label{prop:BS} 
Let $H=H_1\times H_2$, let $\mu_i$ be a probability measure on $H_i$, $i=1,2$ and let $\mu=\mu_1\times \mu_2$. Let $X$ be an $H$-space and $\nu$ a $\mu$-stationary measure on $X$. Then $\nu$ is $\mu_i$-stationary. 
\end{prop}

\begin{proof}
In view of \cite[Corollary 2.7]{Bader-Shalom} it is enough to prove the result when $\nu$ is ergodic. For $\nu$ ergodic this is the statement of \cite[Lemma 3.1]{Bader-Shalom}.
\end{proof}

\begin{lem}\label{lem:discrete-trivial}
%Let $H=H_1\times H_2$ be a semisimple Lie group. Let $\nu$ be a probability measure on discrete subgroups of $H$ such that the projection to $H_1$ is almost surely discrete and the intersection with $H_2$ is almost surely trivial. Suppose farther that $\nu$ is $\mu_{H_2}$ stationary. Then $\nu$ is supported on subgroups of $H_1$.   
Let $H=H_1\times H_2$ be a product of centre-free semisimple Lie groups without compact factors and let $\nu$ be a $\mu_H$-stationary random discrete subgroup. If the projection to $H_1$ is almost surely discrete and the intersection with $H_2$ is almost surely trivial, then $\nu$ is supported on subgroups of $H_1$.
\end{lem}

\begin{proof}
Let us denote by $\pi_i$ the projection to $H_i$ for $i=1,2$.
%Let $\nu$ be a measure as in the statement of the lemma. 
Suppose by way of contradiction that $\nu$ is not supported on subgroups of $H_1$. Then for any sufficiently large ball $B$ in $H_1$ we have 
$$
 \nu\{\Lambda:\exists \ga\in\Lambda~\text{with}~\pi_1(\ga)\in B,\pi_2(\ga)\ne 1\}>0.
$$
Since $\pi_1(\Lambda)$ is discrete a.s., the intersection $\pi_1(\Lambda)\cap B$ is finite a.s. The intersection with $H_2$ is trivial a.s. so the preimage $\Lambda\cap \pi^{-1}(B)$ will be finite as well. Considering the map $\Lambda\mapsto \pi_2(\Lambda\cap\pi_1^{-1}(B))$, the measure $\nu$ induces a $\mu_{H_2}$-stationary finite positive measure on $\mathcal{F}(H_2)$ which is not a combination of $\gd_{\{1\}}$ and $\gd_\emptyset$. This contradicts Corollary \ref{cor:st-fi}.
\end{proof} 

\begin{rem}
Given a locally compact group $G$ and a probability measure $\nu$ on $\sub_d(G)$ it is possible to construct a probability $G$-space $(X,m)$ such that $\nu$ is the pushforward of $m$ via the stabilizer map $X\to \sub(G),~x\mapsto G_x$. This is proven in \cite[Theorem 2.6]{7s} under the assumption that $\nu$ is an IRS, but the proof applies to any $\nu\in\text{Prob}(\sub(G))$ which is supported on unimodular subgroups. Thus, the study of probability measures on $\sub_d(G)$ is equivalent to the study of discrete stabilizers of probability $G$-spaces. We will not make use of this fact.
\end{rem}

%As an immediate consequence of Lemma \ref{cor:nu_2-stationary} and Proposition \ref{prop:BS} we obtain:
%
%\begin{cor}\label{cor:discrete-trivial}
%Let $H=H_1\times H_2$ be a product of semisimple groups and let $\nu$ be a $\mu_H$ stationary random discrete subgroup. If the projection to $H_1$ is almost surely discrete and the intersection with $H_2$ is almost surely trivial, then $\nu$ is supported on subgroups of $H_1$.
%\end{cor}

%%%%%%%%%%%%%%%%%
%%%%%%%%%%%%%%%%%
%%%%%%%%%%%%%%%%%

\section{A decomposition result for Invariant Random Subgroups}\label{sec:IRS}

Let $G$ be a connected centre-free semisimple Lie group without compact factors.
It follows from the Borel density theorem for IRS \cite{7s,GL} that for every ergodic discrete invariant random subgroup $\nu$ on $G$ there is a decomposition $G=G'\times G''$ such that the projection to $G'$ is discrete and Zariski dense and the projection to $G''$ is trivial almost surely.
The following results (see Theorem \ref{thm:decomposition} and Corollary \ref{cor:irreducibility} below) generalize to invariant random subgroups the classical decomposition-to-irreducible-factors theorem for lattices in semisimple groups (see \cite[Theorem 5.22]{Rag}): 

\begin{thm}\label{thm:decomposition}
Let $G=G_1\times\cdots\times G_n$ be a connected centre-free semisimple Lie group without compact factors and with simple factors $G_i,~i=1,\ldots,n$.
Let $\nu$ be an ergodic discrete invariant random subgroup in $G$. Then $G$ decomposes to a product of semisimple factors $G=H_1\times\ldots\times H_k$ with $1\le k\le n$, such that almost surely the projection of a random subgroup to each $H_i$ is discrete while the projection to each proper factor of $H_i$ is dense. 
\end{thm}

The proof relies on the following:

\begin{lem}\label{lem:dense-discrete}
Let $H=H_1\times H_2$ be a product of centre-free semisimple Lie groups without compact factors. Let $\nu$ be an ergodic discrete IRS in $H$ which projects discretely to $H_1$ and Zariski densely to $H_2$. Then the intersection of a random subgroup with $H_2$ is nontrivial almost surely. 
\end{lem}

\begin{proof}[Proof of Lemma \ref{lem:dense-discrete}]
The lemma follows immediately from Lemma \ref{lem:discrete-trivial}.
\end{proof}

\begin{rem}
This lemma could also be proved directly (without referring to the special properties of the measure $\mu_H$) by applying disintegration of measures with respect to the factor map $(\text{Sub}_H,\nu)\to(\sub_{H_1},\pi_*\nu)$. Indeed, if the intersection with $H_2$ is trivial almost surely, then since the projection to $H_2$ is Zariski dense, the $H_2$ action on a generic fibre is free. By
constructing a measurable section for the $H_2$ action on a fibre one may pull the fibre measure to a left $H_2$-invariant probability measure on $H_2$. This is absurd since $H_2$ is not compact.

\end{rem}

\begin{proof}[Proof of Theorem \ref{thm:decomposition}]
Let $G=G_1\times\cdots\times G_n$ and $\nu$ be as in the statement of the theorem. 
We may suppose that a random subgroup is almost surely Zariski dense.
If the projection of a random subgroup to every proper semisimple factor of $G$ is dense almost surely then $\nu$ is irreducible. Otherwise there is a proper decomposition $G=H\times H'$ such that the projection to $H$ is almost surely discrete. We claim that the projection to $H'$ is also discrete and hence we can deduce the result by induction on the number of simple factors. Suppose by way of contradiction that the projection to $H'$ is non-discrete and let $H_2\lhd H'$ be the connected component of its closure. Then $H'$ decomposes as $H'=H''\times H_2$ with $H_2$ nontrivial. It follows that $\nu$ projects discretely to $H_1:=H\times H''$ and densely to $H_2$. Since the intersection of every subgroup of $G$ with $H_2$ is normalized by the projection of that subgroup to $H_2$, a discrete subgroup of $G$ that projects densely to $H_2$ must intersect it trivially. Therefore a $\nu$-random subgroup intersects $H_2$ trivially almost surely. A contradiction to Lemma \ref{lem:dense-discrete}. 
\end{proof}

%\begin{rem}
It follows from Theorem \ref{thm:decomposition} that an ergodic discrete IRS $\nu$ in $G=G_1\times\cdots\times G_n$ is associated with two other IRS $\ti\nu$ and $\overline\nu$ given (respectively) by its projection to and the intersection with the semisimple factors $H_j$. 
Both $\ti\nu$ and $\overline\nu$ are products, $\ti\nu=\prod \ti\nu_j,~\overline\nu=\prod\overline\nu_j$, where $\ti\nu_j$ and $\overline\nu_j$ are discrete IRS on $H_j$.
Obviously $\ti\nu_j$ is irreducible. We claim that also $\overline\nu_j$ is irreducible and nontrivial. If $\nu$ is irreducible then $\nu=\ti\nu=\overline\nu$ and there is nothing to prove. Otherwise, it follows from Lemma \ref{lem:dense-discrete} that $\overline\nu_j$ is nontrivial for every $j$ for which $\ti\nu_j$ is nontrivial.
Fix $j\in\{1,\ldots,k\}$. If $H_j$ is simple then there is nothing further to prove. Suppose that $H_j$ is not simple. Let $F_j\lhd H_j$ be a proper nontrivial semisimple factor of $H_j$.
Let $\gD$ be a $\nu$ random subgroup and let $\ti\gD_j$ and $\overline\gD_j$ be the corresponding random subgroups in $H_j$, that is the projection and the intersection with $H_j$.
We need to show that $F_j\overline\gD_j$ is dense in $H_j$.
Note that $\overline\gD_j\cap F_j$ must be trivial, since it is both discrete and normalized by the projection of $\ti\gD_j$ to $F_j$, which is dense. 
Consider the identity connected component of the closure of $\overline\gD_j F_j$ and denote it by $N_j$. We need to show that $N_j=H_j$ and indeed, if $N_j$ is a proper subgroup then again we deduce that $\overline\gD_j$ is not contained in $N_j$. However, by construction the projection of $\overline\gD_j$ to $H_j/N_j$ is discrete. This is absurd since this nontrivial discrete group is normalized by the dense projection of $\ti \gD_j$ to $H_j/N_j$.
%\end{rem}
Thus we have established the following:

\begin{cor}\label{cor:irreducibility}
Let $G=G_1\times\cdots\times G_n$ be as in Theorem \ref{thm:decomposition}.
Let $\nu$ be an ergodic Zariski dense discrete invariant random subgroup in $G$. Let $G=H_1\times\ldots\times H_k$ be the decomposition of $G$ to $\nu$-irreducible factors. For every $j$ let $\overline\nu_j$ be the IRS in $H_j$ obtain by the intersection of the $\nu$-random subgroup with $H_j$.
Then $\overline\nu_j$ is a non-trivial irreducible discrete IRS of $H_j$. 
\end{cor}

%%%%%%%%%%%
%%%%%%%%%%%
%%%%%%%%%%%

\section{Stationary measures of rank one groups}

The celebrated factor theorem of Nevo and Zimmer \cite{NZ} (Theorem \ref{thm-NZ}) is concerned with stationary measures of higher rank semisimple Lie groups and as shown in \cite[Theorem B]{NZ99} the analog result is not true for simple Lie groups of rank one. In this section we establish the following weak version of Nevo--Zimmer factor theorem for rank one groups.

\begin{thm}\label{lem-Discrete}\label{prop:rank-one}
Let $G$ be a centre-free simple rank-one real Lie group with a smooth probability measure $\mu$. Let $(X,\nu)$ be a non-trivial ergodic probability $\mu$-stationary non-essentially-free $G$-system. Then either $(X,\nu)$ has discrete Zariski dense stabilizers almost surely or there exists a $G$-equivariant map $(X,\nu)\to (G/P,\nu_P)$ where $P$ is a minimal parabolic subgroup of $G$. 
\end{thm}

\begin{lem}\label{lem-Projective}
Let $G$ be a simple real Lie group. Let $V$ be a real linear representation of $G$ without fixed points. Fix a minimal parabolic subgroup $P\subset G$ with Langlands decomposition $P=MAN$. Every ergodic $\mu$-stationary measure $\nu$ on the projective space $\mathbb P^1(V)$ is supported on a single $G$-orbit $G[v]$ where the stabilizer of $[v]$ is a proper subgroup of $G$ that contains $AN$. 
\end{lem}
\begin{proof}
Let $(G/P,\nu_P)$ be the Poisson boundary. We get a $G$-equivariant map $\kappa\colon G/P\to {\rm Prob}(\mathbb P^1(V))$ such that $\int_{G/P} \kappa(gP)d\nu_P(gP)=\nu$. Each $\kappa(gP)$ is a $gPg^{-1}$-invariant measure on $\mathbb P^1(V)$. By Lemma \ref{lem-ProjActions} $\kappa(gP)$ is supported on $gANg^{-1}$ fixed lines almost surely. It follows that almost all points in the support of $\nu$ are stabilized by a conjugate of $AN$. Since $G/AN$ is compact, the $G$-orbits of these points are closed in $\mathbb P^1(V)$. Using ergodicity, we deduce that the measure $\nu$ must be supported on a single $G$ orbit, with stabilizer conjugate to some closed subgroup $H\supset AN$. If the group $G$ fixes a line then it must fix it pointwise, because it is simple. We assumed that $V$ has no fixed subspaces, so $H$ is a proper subgroup.
\end{proof}
\begin{lem}\label{lem-ANParab}
Let $G$ be a connected simple rank one real Lie group. Let $P$ be a minimal parabolic of $G$ with Langlands decomposition $P=MAN$. Let $H$ be closed subgroup containing $AN$. Then either $H\subset P$ or $H=G$.
\end{lem}
\begin{proof}
Let $\frak h, \frak p,\frak m, \frak a,\frak n$ be the Lie algebras of $H,P,M,A,N$ respectively. Let $\Theta$ be a Cartan involution of $G$ stabilizing $A$. Let $\Sigma$ be the root system of $A$ and let $\Sigma^+$ be the set of positive roots.  For any $\lambda\in \frak h^*$ let 
\[\frak g_\lambda=\{X\in \frak g| [Y,X]=\lambda(Y) \textrm{ for } Y\in \frak a\}.\] We have (\cite[p. 122--123]{Knapp})
\begin{equation*}\frak p=\frak g_0+\sum_{\alpha\in \Phi^+}\frak g_\alpha,\qquad\frak n=\sum_{\alpha\in\Phi^+}\frak g_{\alpha},\qquad \frak n^-=\sum_{\alpha\in\Phi^+}\frak g_{-\alpha},\\
\end{equation*} where $\frak n^-$ is the Lie algebra of the unipotent radical of the opposite parabolic.

First consider the case $\frak h\subset \frak p$. Let $H_0$ be the connected component of $H$. We have $H_0\subset P$ and we need to show $H\subset P$. Let $h\in H$. The set of maximal split tori in $H_0$ forms a single conjugacy class \cite[Thm. 15.14]{Bor}, so there exists an $h_0\in H_0$ such that $h_1:=hh_0^{-1}$ normalizes $A$. The rank of $G$ is one so either $h_1$ commutes with $A$ or $h_1ah_1^{-1}=a^{-1}$ for every $a\in A$. In the first case $h_1\in Z_G(A)= MA\subset P$. In the second $h_1 \frak n h_1^{-1}=\frak n^-$, which is impossible because $h_1$ normalizes $\frak h$. We deduce that $h_1\in P$ and consequently $h\in P$. This proves $H\subset P$.

Now consider the case $\frak h\not\subset \frak p$.
 Since $\frak h\not \subset \frak p$, there exists a positive root $\alpha\in \Phi^+$ such that $\frak g_{-\alpha}\cap \frak h\neq 0.$ Let $E_{-\alpha}\in \frak g_{-\alpha}\cap \frak h$ be a non-zero element. For any non-zero $E_{\alpha}\in \mathbb R \Theta(E_{-\alpha})$ the subalgebra $\frak s:=\mathbb R E_{\alpha}+\mathbb R E_{-\alpha}+\mathbb R [E_\alpha,E_{-\alpha}]$ is isomorphic to $\frak{sl}_2(\mathbb R)$ \cite[p. 68]{Knapp}. %By adjusting the scalars we can choose $E_\alpha$ in such a way that $[E_\alpha,Y_\alpha]=-[E_{-\alpha},Y_\alpha]=2$. 
 Note that since $\frak n \subset \frak h$, we automatically have $E_\alpha\in \frak h$, so $\frak s$ is a subalgebra of $\frak h$. Let $Y_\alpha=[E_\alpha,E_{-\alpha}].$ The rank of $G$ is $1$ so $\frak a=\mathbb R Y_\alpha$. %and the root system is either $\Sigma=\{\lambda,-\lambda\}$ or $\Sigma=\{-2\lambda,\lambda,\lambda,2\lambda\}$, with $\lambda$ positive. 

We recall that the representations $\frak s\simeq \frak{sl}_2(\mathbb R)$ have symmetric weight space decomposition with respect to $\frak a$ \cite[Thm 2.4]{Knapp}. For every $\xi\in\frak a^*$ we have $\dim (\frak h\cap \frak g_\xi)=\dim (\frak h\cap \frak g_{-\xi})$. Therefore 
\[\dim \frak n=\dim (\frak h\cap \frak n)=\sum_{\beta\in\Sigma^+} \dim(\frak h\cap \frak g_\beta)=\dim(\frak h\cap \frak n^-)=\dim \frak n^-,\]
so $\frak n^-\subset \frak h$. This proves that $\frak n+\frak n^-\subset \frak h$. 

\textbf{Claim.} $\frak n+\frak n^-$ generates $\frak g$. To prove the claim we pass to the complexification $\frak g_\mathbb C$. It will be enough to show that $\frak n_\mathbb C+\frak n^-_\mathbb C$ generate $\frak g_\mathbb C$. Choose a maximal Cartan subalgebra $\frak b$ of $\frak m$ and let $\frak c=\frak a+i\frak b$. Then, $\frak c_\mathbb C$ is a Cartan subalgebra of $\frak g_\mathbb C$ and all the roots of $\frak c$ in $\frak g_\mathbb C$ are real.  Let $\Phi$ be the root system of $\frak c$ and finally let $V_0\subset \frak c^*$ be the subspace
\[ V_0:=\{\xi\in \frak c^*| \xi(Y_\alpha)=0\}.\] The centralizer of $\frak a$ is $\frak a_\mathbb C+\frak m_\mathbb C$ so the roots of $\frak c$ in $\frak m_\mathbb C$ are precisely those that vanish on $\frak a$. Therefore
\[\frak n_\mathbb C+\frak n^-_\mathbb C=\sum_{\xi\in\Phi\setminus V_0}\frak g_{\mathbb C,\xi}.\]
By \cite[Prop. 4.1.(g)]{Knapp}, we have $\dim \frak g_{\mathbb C,\lambda}=1$ for every root $\lambda\in \Phi$. Moreover, we have \cite[p.61]{Knapp}\footnote{These identities are true only for root systems of complex semisimple algebras, which is why we had to pass to the complexification of $\frak g$.}
\begin{equation}\label{eq-comutators}
[\frak g_{\mathbb C,\beta},\frak g_{\mathbb C,\gamma}]=\begin{cases} \frak g_{\mathbb C,\beta+\gamma} &\textrm{if } \beta+\gamma\in \Phi,\\
\mathbb C X_\beta &\textrm{if } \gamma=-\beta,\\
0 &\textrm{otherwise,}\end{cases}
\end{equation}
where $X_\beta$ is an element of $\frak c$ such that $\xi(X_\beta)=\langle \xi,\beta\rangle$ for every $\xi\in \frak c^*$. By Lemma \ref{lem-RootGen}, $\Phi\setminus V_0$ generates $\Phi$. By (\ref{eq-comutators}) the set of roots in the Lie algebra generated by $\frak n_\mathbb C+\frak n_\mathbb C$ is closed under taking reflections so it must be the whole root system. Using (\ref{eq-comutators}) once again we deduce that $\frak n_{\mathbb C}+\frak n^-_\mathbb C$ generates $\frak g_\mathbb C$. The claim is proved.

It follows that $\frak h=\frak g$, so $H=G$.
\end{proof}

\begin{proof}[Proof of Theorem \ref{prop:rank-one}]
Assume the stabilizers are not discrete almost surely. By ergodicity, the dimension of the Lie algebra of $\overline{}{\rm Stab_G}(x)$ is almost surely equal to some constant $k$. The action is non-trivial so $k<\dim G$.  Let $\mathcal G_{k,\frak g}$ be the Grassmannian of $k$-planes in $\frak g$ and let \[\pi\colon \mathcal G_{k,\frak g}\ni W\to \left[ \bigwedge\nolimits^k W\right]\in \mathbb P^1\left(\bigwedge\nolimits^k \frak g\right)\] be the Pl\"ucker embedding. The image of the Grassmannian is a closed subset of $\mathbb P^1(\bigwedge^k \frak g)$. The map $x\mapsto \pi({\rm lie}(\overline{{\rm Stab_G x}}))\in \mathbb P^1(\bigwedge^k \frak g)$ is $G$-equivariant. Write $V=\bigwedge^k \frak g$ and let $\nu'$ be the pushforward of $\nu$ to $\mathbb P^1(V)$. By Lemma \ref{lem-Projective}, $(\mathbb P^1(V), \nu')$ is supported on a single orbit $G[v]$, with $H:={\rm Stab_G}([v])\supset AN$. The group $G$ is simple, so no $k$-dimensional subspace of $\frak g$ can be fixed by $G$. Since the orbit $G[v]$ is contained in the image of the Pl\"ucker embedding we deduce that $H$ is a proper subgroup of $G$. By Lemma  \ref{lem-ANParab} $H$ is a subgroup of $P$, so $(\mathbb P^1(V), \nu')$ admits $(G/P,\nu_P)$ as a factor. 

Assume now that ${\rm Stab}_G(x)$ is not  Zariski dense almost surely. The stabilizers must be infinite because by Lemma \ref{lem-StatioFin} there are no $\mu$-stationary measures on the conjugacy classes of finite subsets of $G\setminus \{1\}$. Therefore the stabilizers must have Zariski closures of dimension between $1$ and $\dim \frak g-1$. By ergodicity, there exists $1\leq k< \dim \frak g$ such that the Zariski closure $\overline{{\rm Stab}_Gx}^{Z}$ is $k$-dimensional almost surely. Consider the map $x\mapsto \pi({\rm lie}(\overline{{\rm Stab_G x}}^Z))\in \mathbb P^1(\bigwedge^k \frak g)$. Arguing as in the paragraph above we prove that $(X,\nu)$ admits $(G/P,\nu_P)$ as a factor. 
\end{proof}

\section{Stiffness of discrete stationary random subgroups for higher rank groups}\label{sec:stiffness}

In this section we establish a stiffness result for stationary measures on the space of discrete subgroups. In particular we show that every such measure which is `irreducible' with respect to the rank one factors of $G$ is invariant (see Theorem \ref{thm:stiffness}). This result is a consequence of a decomposition theorem (Theorem \ref{thm:stationary-decomposition}) which extends the results of \S \ref{sec:IRS} from invariant to stationary measures. 
%
% a decomposition result (Theorem \ref{thm:stationary-decomposition}) for discrete stationary random subgroups of semisimple groups, and deduce a stiffness result (Theorem \ref{thm:stiffness}) and and a variant of the Stuck--Zimmer theorem (Theorem \ref{thm:stationary-Stuck--Zimmer}) for stationary measures.
The key ingredient in the proof of Theorem \ref{thm:stationary-decomposition} is the following result, due to Nevo and Zimmer \cite{NZ}.

\begin{thm}[Nevo-Zimmer \cite{NZ}]\label{thm-NZ}
Let $G$ be a higher rank semisimple Lie group. Let $\mu$ be a smooth probability measure on $G$ and let $(X,\nu)$ be a probability $\mu$-stationary action of $G$. Then either
\begin{itemize}
\item  $\nu$ is $G$-invariant, 
\item there exists a proper parabolic subgroup $Q\subset G$ and a measure preserving $G$-equivariant map\footnote{ By measure preserving we only mean that $\nu(\pi^{-1}(A))=\mu_Q(A)$ for every Borel set $A\subset G/Q$. In particular we do not require that it is a relatively measure preserving extension.} $\pi\colon (X,\nu)\to (G/Q,\mu_Q),$ where $\mu_Q$ is the unique $\mu$-stationary measure on $G/Q$, or
\item $(X,\nu)$ has an equivariant factor space, on which $G$ acts via a rank one factor group.
\end{itemize}
\end{thm}

%We will also need some simple lemmas on the discrete random subgroups of parabolic groups. 
In what follows, $P$ will be a minimal parabolic subgroup of $G$.

\begin{lem}\label{lem-AIRS} Let $Q$ be a parabolic subgroup of $G$. Let $Q=LN$ be a Levi decomposition of $Q$ and let $A_L$ be the centre of $L$. Then, any discrete $A_L$-invariant random subgroup of $Q$ is contained in $L$ almost surely. 
\end{lem}

\begin{proof}

Let $a\in A_L$ be an element such that the norm of the restriction of $\text{Ad}(a^{-1})$ to the Lie algebra of $N$ is less than $1$. It follows that for any two identity neighbourhoods $U_1,U_2\subset N$ in $N$ such that $U_2$ is bounded there is $n_0\in\BN$ such that $\text{Ad}(a^{-n_0})(U_2)\subset U_1$.

Let $\lambda$ be a discrete $A_L$-invariant random subgroup of $Q$. Suppose by way of contradiction that $\lambda (\{\gC\subset L\})<1$. Then there is a bounded identity neighbourhood $V_2$ in $Q$ such that 
$\lambda(\gO)>0$ where $\gO:=\{\gC:\gC\cap V_2\setminus L\ne\emptyset\}$.
Furthermore, we may take $V_2$ to be of the form $V_2=W\cdot U_2$ where $W\subset L$ and $U_2\subset N$ and suppose that they are preserved by $a^{-1}$ (e.g. we can suppose that $\log U_2$ is a norm ball in the Lie algebra of $N$). Since $\lambda$ is discrete we can chose a small identity neighborhood $U_1\subset N$ so that setting $V_1=W\cdot U_1$ we have 
$$
 %\lambda(\{\gC:\gC\cap (V_2\setminus L)\ne\emptyset,~\text{and}~:\gC\cap (V_1\setminus L)=\emptyset \}:=\gep>0.
 \gep:=\lambda(\{\gC\in\gO:\gC\cap (V_1\setminus L)=\emptyset \})>0.
$$
Choosing $n_0$ as above we get that $\lambda (\gO^{a^{n_0}})\le\lambda(\gO)-\gep$ in contrast with the assumption that $\lambda$ is $A_L$-invariant.
\end{proof}
\begin{lem}\label{lem-LeviInt}
Let $G$ be a centre-free complex semisimple Lie group, let $Q$ be a proper parabolic subgroup. The intersection of all Levi subgroups of $Q$ is the product of the simple factors of $G$ contained in $Q$.
\end{lem}

\begin{proof}
We can quotient $G$ by the product of all simple factors contained in $Q$. In this way we can assume that $Q$ contains no simple factors of $G$. Every parabolic subgroup in $G$ is a product of parabolic subgroups in the simple factors, so the lemma reduces to the case where $G$ is simple. 

Assume from now on that $G$ is simple. Let $J$ be the intersection of all Levi subgroups of $Q$. We need to show that $J=\{1\}$. Let $L$ be a Levi subgroup of $Q$ and let $N$ be the unipotent radical of $Q$. Let $P\subset Q$ be a Borel subgroup with a Levi subgroup $A\subset L$. Then $A$ is a maximal torus of $G$ (note that $G$ is split since we work over the field of complex numbers). Write $\frak{g,q,l,n,p,a}$ for the corresponding Lie algebras. Let $\Phi$ be the root system of $\frak g$ with respect to $\frak a$. Let $\Phi^+$ be the set of positive roots corresponding to $P$ and let $\Pi\subset \Phi^+$ be the subset of simple roots. Let $W$ be the Weyl group. Finally let $F\subset \Pi$ be the subset of simple roots lying in $\frak q$. Write $E\subset \frak a^*$ for the space spanned by $F$ and let $S=\Phi\cap E$. Then 
\begin{equation*} 
\frak l=\sum_{\lambda\in S} \frak g_\lambda,\qquad\frak n=\sum_{\substack{\lambda\in \Phi^+\setminus S}} \frak g_\lambda,\qquad\frak q=\frak l+\frak n=\sum_{\lambda\in S\cup \Phi^+}\frak g_\lambda,
\end{equation*}
see e.g. \cite[p.67]{GV}. Let $R:=\Phi^+\setminus S$. By Lemma \ref{lem-RootGen}, $R$ spans $\frak a^*$. Let $A'\subset A$ be a maximal torus of $J$. Since both $J$ and $N$ are normal in $Q$ we have $[J,N]\subset J\cap N\subset L\cap N=\{1\}$. The torus $A'$ commutes with $N$ so the roots in $R$ vanish on its Lie algebra $\frak a'$. As $R$ spans $\frak a^*$ we get $\frak a'=0$. As a normal subgroup of $L$, $J$ must be reductive, so triviality of maximal tori implies that $J$ is finite. Since $J$ is finite and normal in $P$ it must be in the centre, so $J=\{1\}$. 
\end{proof}

\begin{lem}\label{lem-PIRS} Let $G$ be a centre-free semisimple Lie group with proper parabolic subgroup $Q$ containing $P$. Let $\Lambda\subset Q$ be a discrete $P$-invariant random subgroup of $Q$. Then there is a proper semisimple factor $H$ of $G$ such that $\Lambda$ is supported on discrete subgroups of $H$ almost surely. 
\end{lem}

\begin{proof}
Let $L\le Q$ be a Levi subgroup of $Q$, such that the centre $A$ of $L$ is contained in $P$. By Lemma \ref{lem-AIRS}, $\Lambda\subset L$ almost surely. Since $\Lambda$ is $P$-invariant and $Q=LP$ we can upgrade this inclusion to $\Lambda\subset \bigcap_{p\in P} L^p=\bigcap_{q\in Q} L^q$. Let $J$ be the intersection of all Levi subgroups of $Q$. The set of real points of $Q$ is Zariski dense so $J$ coincides with the real points of the intersection of all complex Levi subgroups of $Q(\mathbb C)$. By Lemma \ref{lem-LeviInt} we deduce that $\Lambda$ is supported on $\sub_d(H)$ where $H=\bigcap_{q\in Q} L^q$ is a proper semisimple factor of $G$ almost surely.
\end{proof}

%%%%%%%%%%%%%%%%%
%%%%%%%%%%%%%%%%%
%%%%%%%%%%%%%%%%%

The following decomposition theorem generalizes Theorem \ref{thm:decomposition} from invariant random subgroups to discrete stationary random subgroups:

\begin{thm}\label{thm:stationary-decomposition}
Let $G$ be a connected centre-free semisimple Lie group without compact factors and $\nu$ a discrete $\mu$-stationary random subgroup of $G$. Then $G$ decomposes to a product of three semisimple factors $G=G_\mathcal{I}\times G_\mathcal{H}\times G_\mathcal{T}$ such that
\begin{enumerate}
\item $\nu$ projects to an IRS in $G_\mathcal{I}$ for which all the irreducible factors are of rank at least $2$.
\item $G_\mathcal{H}$ is a product of rank one factors and $\nu$ projects discretely to every factor of $G_\mathcal{H}$.
\item $\nu$ projects trivially to $G_\mathcal{T}$.
\end{enumerate}
Furthermore, the intersection of a random subgroup with every simple factor of $G_\mathcal{H}$ as well as with every irreducible factor of $G_\mathcal{I}$ is  almost surely Zariski dense in that factor.
\end{thm}

By the irreducible factors of $G_\mathcal{I}$ we mean the irreducible factors associated with the decomposition of IRS as in Theorem \ref{thm:decomposition}. The subscripts $\mathcal{I}, \mathcal{H},\mathcal{T}$ stands for invariant, hyperbolic and trivial (respectively).

\begin{proof}[Proof of Theorem \ref{thm:stationary-decomposition}]
We will argue by induction on the number of simple factors of $G$. If $G$ is simple 
then we may suppose $\text{rank}(G)\ge 2$ since for rank one groups the statement trivially holds. 
If $\nu$ is invariant then it is either trivial or Zariski dense by the Borel density theorem for IRS. 

Suppose by way of contradiction that $G$ is simple of higher rank and $\nu$ is not invariant. 
Let $\pi\colon (\sub(G),\nu)\to (G/Q,\nu_Q)$ be the measure preserving $G$-equivariant map afforded by Theorem \ref{thm-NZ}. Let $P\subset Q$ be a minimal parabolic. 

Let $\mathcal P(\sub(G))$ be the space of Borel probability measures on $\sub(G)$. 
%As a first step we will construct a $G$-equivariant map 
%\[ \kappa\colon (G/P,\nu_P)\to \mathcal P(\sub(G)),\] such that $\nu=\int_{G/P} \kappa(gP) d\nu_P(gP).$ 
Consider the map $\psi\colon G\to \mathcal P(\sub(G))$ given by $\psi(g):=g_* \nu$. Then, $\psi$ is a $G$-equivariant, bounded $\mu$-harmonic function on $G$. The pair $(G/P,\nu_P)$ is the Poisson boundary, so there exists a unique measurable function $\kappa\colon (G/P,\nu_P)\to \mathcal P(\sub(G))$ such that $\nu=\int_{G/P}\kappa(gP)d\nu_P(gP).$
The uniqueness implies that the map is $G$-equivariant. This also implies that $\kappa(gP)$ is $gPg^{-1}$ invariant $\nu_{P}$-almost surely. 

Consider the composition $\pi_*\circ \kappa\colon (G/P,\nu_p)\to \mathcal P (G/Q).$ This map is $G$-equivariant so the measure $\pi_*\circ \kappa(gP)$ is $gPg^{-1}$-invariant almost surely. We claim that the unique probability measure on $G/Q$ with this property is $\delta_{gQ}$. To justify that let use prove it in the case $g=1$. Let $L$ be a Levi subgroup of $Q$ and let $W_L$ the Weyl group of $L$. Let $N$ be the unipotent radical of $P$. By the Bruhat decomposition, the orbits of $N$ on $G/Q$ correspond to the cosets $W/W_L$. The orbit corresponding to $w W_L$ is measurably isomorphic to $N/(N\cap Q^w)$. For all cosets except the trivial one, $N\cap Q^w\neq N$. Since $N$ is nilpotent, this means that $N/[N,N](N\cap Q^w)$ is also non-trivial. The action of $N$ on the last quotient is just by translations on some power of $\mathbb R$ so obviously it does not support any invariant probability measures. It follows that only the trivial orbit can support a $P$-invariant measure. We deduce that $\pi\circ \kappa(gP)=gQ$ almost surely. 

By comparing the stabilizers we deduce that for $\nu_P$ almost every $gP\in G/P$, the measure $\kappa(gP)$ is supported on the set of discrete subgroups of $gQg^{-1}$. The measure $\kappa(gP)$ is $gPg^{-1}$-invariant, so it is the distribution of a $gPg^{-1}$-invariant random discrete subgroup of $gQg^{-1}$. By Lemma \ref{lem-PIRS} and simplicity of $G$ we get that $\kappa(gP)=\delta_{\{1\}}$ almost surely. Since $\nu=\int_{G/P} \kappa(gP) d\nu_P(gP)$ we deduce that $\nu=\gd_{\{1\}}$, a contradiction. This completes the proof of the theorem when $G$ is simple.

Suppose now that $G=\prod_{i=1}^nG_i$ has $n>1$ factors. 
In view of Corollary \ref{cor:same-stationary} we may suppose that $\mu=\mu_1\times\cdots \times\mu_n$ where $\mu_i$ is a probability measure on $G_i$.
If $\nu$ is invariant the theorem follows from Theorem \ref{thm:decomposition} and Corollary \ref{cor:irreducibility}. If $\nu$ admits a parabolic factor $(G/Q,\nu_Q)$ then arguing as in the simple case above we deduce from Lemma  \ref{lem-PIRS} that $\nu$ is supported on discrete subgroups of $H$ for some proper semisimple factor $H\lhd G$. Then we deduce the result from the induction hypothesis. 

Thus we are left with the case where $\nu$ is not invariant and does not admit a $G/Q$-factor. Then it follows from Theorem \ref{thm-NZ} and Theorem \ref{lem-Discrete} that for some rank one factor, say, $G_1$, the projection of a $\nu$-random subgroup is almost surely discrete and Zariski dense in $G_1$. Consider the intersection with $H=G_2\times\ldots\times G_n$. Since the intersection is a $\mu_2*\cdots *\mu_n$ stationary random subgroup it follows from the induction hypothesis that there is a decomposition $H=H_\mathcal{I}\times H_\mathcal{H}\times H_\mathcal{T}$ as in the statement of the theorem. 
Let $\gD$ denote the $\nu$-random subgroup.
Since the intersection $\gD\cap (H_\mathcal{I}\times H_\mathcal{H})$ is almost surely Zariski dense in $H_\mathcal{I}\times H_\mathcal{H}$ and normalized by the projection of $\gD$ to $H_\mathcal{I}\times H_\mathcal{H}$, it follows that this projection is discrete as well. Thus the projection of a $\nu$-random subgroup to $G_1\times H_\mathcal{I}\times H_\mathcal{H}$ is discrete almost surely. Applying Lemma \ref{lem:discrete-trivial} we deduce that a random subgroup projects trivially to $H_\mathcal{T}$. Furthermore, by the same reasoning we see that the projection of $\gD$ to every irreducible factor of $H_\mathcal{I}$ as well as to every simple factor of $H_\mathcal{H}$ is discrete. 
%Set $F=H_I\times H_S$ and let $\ti\gD$ be the projection of $\gD$ to $F$. Consider the decomposition $F=F_I\times F_S$ associated with the discrete stationary subgroup $\ti\gD$. Then $F_I$ and $H_I$ have the same irreducible factors of rank $\ge 2$. Note also that $F_S\supset H_S$ and perhaps some rank one simple factors of $H_I$ are now in $F_S$. 
Finally, consider the intersection of the random subgroup with $G_1$. By Lemma \ref{lem:discrete-trivial} this intersection is non-trivial and since it is normalized by the projection to $G_1$ which by assumption is Zariski dense we deduce that the intersection with $G_1$ is also Zariski dense.
Thus our desired decomposition is given by 
$$
 G_\mathcal{I}=H_\mathcal{I},~G_\mathcal{H}=G_1\times H_\mathcal{H},~G_\mathcal{T}=H_\mathcal{T}.
$$
\end{proof}

We are now in a position to deduce:

\begin{thm}[Theorem \ref{thm:stiffness} of the introduction]\label{thm:stiff}
Let $G$ be a connected centre-free semisimple Lie group without compact factors and real rank at least two. Let $\nu$ be a $\mu$-stationary measure on $\sub_d(G)$. Suppose that $\nu$-almost every subgroup intersects trivially every rank one factor of $G$.
%or the projection to the factor is non-discrete. 
Then $\nu$ is invariant. 
\end{thm}

\begin{proof}[Proof of Theorem \ref{thm:stiffness}]
The assumption on the rank one factors implies that in the decomposition of $G$ according to $\nu$ given by Theorem \ref{thm:stationary-decomposition} the factor $G_\mathcal{H}$ is trivial.
\end{proof}

\begin{rem}
The assumption that for every rank one factor the intersection is trivial can be replaced by the assumption that the intersection is not Zariski dense in that factor almost surely. This applies also to Theorem \ref{thm:S-SZ} below.
\end{rem}

We remark that the higher rank assumption in Theorem \ref{thm:stiffness} is necessary.
The following construction demonstrates how to produce a stationary non-invariant measure on the space of discrete subgroups for certain rank one simple groups:\footnote{Nevo and Zimmer constructed stationary actions of rank one groups on compact spaces which admit no invariant probability measure \cite[Theorem B]{NZ99}. However in their construction, the pushforward of every stationary measure to the space of discrete subgroups via the stabilizer map collapses to an invariant measure on $\sub_d(G)$.}

\begin{exam}\label{exam:rank-1}

Denote by $T=\text{Cay}(F_2,\{a_1,a_2\})$ the Cayley graph of the free group $F_2$ with respect to a free basis $\{a_1,a_2\}$. Then $T$ is a decorated $4$-regular tree. Let $\zeta\in\partial T$ be a point at the visual boundary  and let $X_\zeta$ be the grandfather graph associated to $\zeta$. That is, $X_\zeta$ is obtained from $T$ by adding edges between any two points of distance $2$ belonging to a ray converging to $\zeta$. Observe that $X_\zeta$ is a $14$-regular graph obtained from $T$ by adding two disjoint $10$-regular trees. By adding proper labelings to the new edges we may decorate $X_\zeta$ as a Schreier graph of $F_7$ modulo some subgroup $H$, extending the given decoration of $T$ (as a Cayley graph of $F_2\le F_7$). Observe also that the corresponding subgroup $H$ is confined in $F_7$.
Let $\mathcal{C}$ denote the closure of the conjugacy class of $H$ in $F_7$. It is easy to see that $\mathcal{C}$ consists of groups whose Schreier graph is also a decoration of  $X_\zeta$, for some $\zeta\in\partial T$,  extending the given decoration of $T$. Since $X_\zeta$ `remembers' $\zeta$, we obtain a canonical continuous map $\phi:\mathcal{C}\to \partial (T)$ sending $H'\in \mathcal{C}$ to the point at infinity associated to the grandfather graph obtained by forgetting the labels of the Schreier graph of $F_7/H'$. Given an invariant random subgroup of $F_7$ supported on $\mathcal{C}$, the push-forward $\phi_*\nu$ would be a measure on $\partial (T)$ which is (in particular) $F_2$-invariant. Since $F_2$ acts minimally and proximally on its boundary, no such measure exists. It follows that the closed $F_7$ invariant set $\mathcal{C}$ supports no invariant probability measure. As pointed out to us by one of the referees, similar constructions were given also by Glasner and Weiss (see \cite[Example 3.3]{GW}). 

Similarly let $\gC$ be a surface group of genus $7$, let $f:\gC\to F_7$ be a surjective homomorphism and set $\gL=f^{-1}(H)$. It follows that there is no IRS of $\gC$ supported on the closure of the conjugacy class of $\gL$.
Realising $\gC$ as a uniform lattice in $\SL(2,\BR)$, one can show that the closure of the conjugacy class of $\gL$ in $\sub(\SL(2,\BR))$ supports no invariant probability measures. To see this note that the homogeneous space $SL(2,\BR)/\gC$ is an invariant factor, and the disintegration of the measure would produce invariant measures on the fibres, a contradiction (cf. \cite[Lemma 6.1]{NZ99}). In particular, every weak-$*$ limit of the Cesaro averages 
$\frac{1}{n}\sum_{i=0}^{n-1}\mu_G^{(i)}*\gd_{\gL}$
is a stationary measure on $\sub_d(\SL(2,\BR))$ which cannot be invariant. Similar constructions can be made for every rank one group that admits a lattice which projects on $F_2$.

%\medskip
%
% Let $T$ be a $3$ regular tree and consider $X$ to be its grandfather graph according to a chosen point at the visual boundary $\partial T$. That is, $X$ is obtained from $T$ by adding edges between any two points of distance 2 which lie on a ray to the chosen point at infinity. Then $X$ is an $8$ regular graph, obtained from $T$ by adding two disjoint $5$-regular trees and is hence $8$ colorable using $3$ colors for $T$ and another $5$ for the new edges. 
%By doubling every edge and giving it the two possible orientations we obtain a Schreier graph $F_8/H$ associated with some `confined' subgroup $H\le F_8$.
%Note that the closure of the conjugacy class $\mathcal{C}$ of $H$ in $F_8$ consists of groups whose Schreier graph is a similar such oriented colouring of $X$. 
%Since the graph $X$ `remembers' the chosen point at the boundary of $T$, we deduce that there is no $F_8$-invariant probability measure on $\mathcal{C}$. Indeed $F_8$ 
%acts transitively on $T$ and hence preserves no probability measure on $\partial(T)$. 
%%is mapped onto the transitive lattice in $\aut(T)$ that preserves the colouring of  
%Embed $F_8$ as a lattice in $\SL(2,\BR)$ and let $\gL$ be the image of $H$. Then every weak-$*$ limit of the random walk 
%$\frac{1}{n}\sum_{i=0}^{n-1}\mu_G^{(i)}*\gd_{\gL}$
%is a stationary measure on $\sub_d(\SL(2,\BR))$ which in view of \cite[Lemma 6.1]{NZ99} cannot be invariant. Similar constructions can be made for every rank one group that admits a lattice which projects on $F_2$.
\end{exam}

%%%%%%%%%%%%%%
%%%%%%%%%%%%%%
%%%%%%%%%%%%%%

\section{Stuck--Zimmer theorem for stationary measures}
 
%The aim of this section is to extend the celebrated Stuck--Zimmer theorem to stationary measure with discrete stabilizers.

The decomposition result, Theorem \ref{thm:stationary-decomposition}, allows to extend the celebrated Stuck--Zimmer theorem to stationary measure with discrete stabilizers.

\begin{thm}[Theorem \ref{thm:stationary-Stuck--Zimmer} of the introduction]\label{thm:S-SZ}
Let $G$ be a connected centre-free semisimple Lie group without compact factors. Suppose that $G$ has real rank at least $2$ and Kazhdan's property (T). Let $\nu$ be an ergodic $\mu$-stationary measure on $\sub_d(G)$. Suppose that $\nu$-almost every random subgroup intersects trivially every rank one factor of $G$. Then there is a semisimple factor $H\lhd G$ and a lattice $\gC\le H$ such that $\nu=\nu_\gC$. 
\end{thm}

We will rely on the following proper ergodicity result:

\begin{lem}\label{lem:G1dense->G2-Invartant} 
Let $G=G_1\times G_2$ where $G_1,G_2$ are locally compact second countable groups. Let $X$ be a compact, second countable $G$-space with $G_1$-invariant probability measure $\nu$ such that ${\rm Stab}_G (x)$ is discrete and has a dense projection onto $G_2$ for $\nu$-almost every $x\in X$. Then $\nu$ is $G$-invariant. 
\end{lem}

\begin{proof}
In view of the ergodic decomposition of probability measure preserving actions, the lemma clearly reduces to the ergodic case, so let us assume that $\nu$ is $G_1$-ergodic. Using the large stabilizers of the action and Kakutani's ergodic theorem for random walks we will prove that such measure must be also $G_2$-invariant.

Choose a smooth symmetric probability measure $\eta$ on $G_1$ whose support generates $G_1$. Write $\eta^{g_1}$ for the measure $\eta^{g_1}(A)=\eta(g_1^{-1}A g_1)$. Since the support of $\eta^{g_1}$ generates $G_1$, the measure $\nu$ is $\nu^{g_1}$ stationary and ergodic with respect to the random walk on $X$ induced by $\eta^{g_1}$, for every $g_1\in G_1$. We say that a point $x\in X$ is $\eta^{g_1}$-generic for $\nu$ if
\begin{equation}\label{eq-Birkhoff}
\lim_{n\to\infty}\frac{1}{n} \sum_{i=1}^{n-1} \int f(g x)d(\eta^{g_1})^{\ast i}(g)=\lim_{n\to\infty}\frac{1}{n} \sum_{i=1}^{n-1} \int f(g_1^{-1}g g_1 x)d\eta^{\ast i}(g)=\int f d\nu(x),
\end{equation} for all continuous functions $f$. By Kakutani's pointwise ergodic theorem for random walks \cite{Kak51}, the set of $\eta^{g_1}$-generic points has full measure with respect to $\nu$. It follows that the set \[\{ (g_1,x)| x \textrm{ is not } \eta^{g_1} \textrm{-generic}\}\subset G_1\times X\] has zero measure. By Fubini's theorem we deduce that there is a subset $X'\subset X$ with $\nu(X')=1$, such that every $x\in X'$ satisfies the following property. For almost every $g_1\in G_1$  
\begin{equation}\label{eq-ErgAv}
\lim_{n\to\infty}\frac{1}{n} \sum_{i=1}^{n-1}  (\eta^{g_1})^{\ast i}\delta_{x}=\nu.
\end{equation} in the weak-* topology. 

Take a point $x\in X'$ such that $\Gamma:={\rm Stab}_G(x)$ is discrete and has a dense projection onto $G_2$. Since $\Gamma$ is countable, we may fix  $g_1\in G$ such that (\ref{eq-ErgAv}) holds for $g_1\gamma_1$, for every $\gamma=(\gamma_1,\gamma_2)\in \Gamma$. Note that $\eta^{\gamma g}=\eta^{\gamma_1g_1}$. Comparing (\ref{eq-ErgAv}) for $g_1$ and $\gamma_1 g_1$ and using the fact that $\gamma^{-1}x=x$ we find that \begin{align*}\nu=&\lim_{n\to\infty}\frac{1}{n} \sum_{i=1}^{n-1}  (\eta^{\gamma_1g_1})^{\ast i}\delta_{x}=\lim_{n\to\infty}\frac{1}{n} \sum_{i=1}^{n-1}  (\eta^{\gamma g_1})^{\ast i}\delta_{x}\\=&\lim_{n\to\infty}\frac{1}{n} \sum_{i=1}^{n-1}  \gamma_* (\eta^{g_1})^{\ast i}\delta_{\gamma x}=\gamma_*\lim_{n\to\infty}\frac{1}{n} \sum_{i=1}^{n-1}  (\eta^{g_1})^{\ast i}\delta_{x}=\gamma_*\nu.\end{align*} This means that $\Gamma G_1\subset {\rm Stab}_G \nu$. The action of $G$ on the set of probability measures on $X$ is continuous, so ${\rm Stab}_G \nu\supset \overline{\Gamma G_1}=G.$
\end{proof}

%\begin{cor}\label{cor:G1dense->G2-Invartant} 
%Let $G_1$ and $G_2$ be locally compact second countable groups.
%A discrete random subgroup of $G_1\times G_2$ which is $G_1$ invariant and has dense projection onto $G_2$ must be $G_1\times G_2$ invariant.
%\end{cor}
%
%\begin{proof}
%We only need to show that every $G_1$-invariant random subgroup of $G$ can be realized as a stabilizer of a random point in a continuous action $G\curvearrowright X$ with a $G_1$-invariant probability measure $\nu$. The proof is a minor modification of the proof that every IRS is given as a stabilizer in a measure preserving action \cite[Thm 2.6]{7s}. Let $K(G)$ be the space of closed subsets of $G$ endowed with the topology of Gromov-Hausdorff convergence on compact sets. Then $K(G)$ is compact and the natural action of $G$ is continuous. The proof of \cite[Thm 2.6]{7s} yields a probability measure on ${\sub(G)}\times {\mathcal S}(G)$ where $\mathcal S(G)$ is the space of discrete subsets of $G$. The latter is a measurable subset of $\sub(G)\times K(G)$. We just need to observe that the construction still works if we replace the $G$-invariant distribution on ${\sub(G)}$ by a $G_1$-invariant one. 
%\end{proof}

\begin{cor}\label{cor:ergodicity}
Let $G_1$ and $G_2$ be locally compact second countable groups.
Let $\nu$ be an ergodic discrete invariant random subgroup in $G_1\times G_2$ with dense projections to $G_2$. Then $\nu$ is $G_1$-ergodic.
\end{cor}

\begin{proof}
Let $\nu$ be an ergodic discrete invariant random subgroup in $G_1\times G_2$ with dense projections to $G_2$.
In view of \cite[Theorem 2.6]{7s} $\nu$ can be realised as the stabilizer of some compact, second countable $G$-space $X$. Abusing notations we will denote the measure on $X$ by $\nu$ as well. Consider the ergodic decomposition of $(X,\nu)$ with respect to the $G_1$-action.
By Lemma \ref{lem:G1dense->G2-Invartant}  the $G_1$-ergodic components of $\nu$ are $G_2$-invariant. Since $\nu$ is $G_1\times G_2$-ergodic it follows that the decomposition is trivial. This means that $\nu$ is $G_1$-ergodic.
\end{proof}

Thus we have established:

\begin{cor}\label{cor:properly-ergodic}
Under the assumptions of Corollary \ref{cor:irreducibility}, the IRS $\overline{\nu}_j$ are irreducible, that is, $\overline{\nu}_j$ is ergodic with respect to every simple factor of $H_j$.
\end{cor}

Recall the main result of \cite{SZ} (see also \cite[\S 4]{7s}):

%\begin{thm}[Stuck--Zimmer]
% Let $G$ be a centre-free semisimple Lie group of real rank at least $2$ and with Kazhdan's property $(T)$. Let $\nu$ be a non-atomic irreducible IRS of $G$. Then $\nu=\nu_\gC$ for some irreducible lattice $\gC\le G$.
%\end{thm}
 
\begin{thm}[Stuck--Zimmer \cite{SZ}]
Let $G$ be a centre-free semisimple Lie group of real rank at least $2$ and with Kazhdan's property $(T)$.
%Let $G$ be a connected semisimple Lie group without centre, such that each simple factor of $G$ has $\mathbb{R}$-rank at least 2 or is isomorphic to $\text{Sp}(n,1),~n\ge 2$ or $F_{4(-20)}$. 
Suppose that $G$, as well as every rank one factor of $G$, acts ergodically and faithfully preserving a probability measure on a space $X$. Then there is a normal subgroup $N \lhd G$ and a lattice $\Gamma < N$ such that for almost every $x \in X$ the stabilizer of $x$ is conjugate to $\Gamma$.
\end{thm}

The following is a straightforward conclusion of the combination of Theorem \ref{thm:decomposition}, Corollary \ref{cor:irreducibility}, Corollary \ref{cor:properly-ergodic}, and the
Stuck Zimmer theorem:

\begin{prop}\label{prop:IRS-lattices}
Let $G=G_1\times\cdots\times G_n$ be a connected centre-free semisimple Lie group without compact factors. Suppose that $G$ has Kazhdan's property (T).
Let $\nu$ be an ergodic Zariski dense discrete invariant subgroup in $G$, and suppose that the projection of a $\nu$-random subgroup to every rank one simple factor of $G$ is non-discrete. Then $\nu$ is supported on lattices in $G$.
\end{prop} 

\begin{rem}\label{rem:commensurable}
Note that in the context of Proposition \ref{prop:IRS-lattices} the three invariant random subgroups $\nu,\ti\nu$ and $\overline\nu$ introduced after Theorem \ref{thm:decomposition} are commensurable in the obvious sense.
\end{rem}

Relying on \cite[Corollary 1.4]{HT} we can also formulate the following variant of Proposition \ref{prop:IRS-lattices}
\begin{prop}\label{rem:Hartman}
Let $G=G_1\times\cdots\times G_n$ be a connected centre-free semisimple Lie group without compact factors. Suppose that one of the $G_i$ has Kazhdan's property (T).
Let $\nu$ be an ergodic Zariski dense discrete invariant subgroup in $G$ and suppose that $\nu$ is irreducible in the sense of Theorem \ref{thm:decomposition}.
Then $\nu$ is supported on lattices in $G$.
\end{prop}

We can now complete the proof of Theorem \ref{thm:stationary-Stuck--Zimmer}.
Suppose that $\nu$ is not $\gd_{\{1\}}$.
As in the proof of Theorem \ref{thm:stiffness} we get that $G=G_\mathcal{I}\times G_\mathcal{T}$. Since all the irreducible factors of $G_\mathcal{I}$ are of rank at least $2$, the result follows from Proposition \ref{prop:IRS-lattices}.
\qed

\section{Margulis conjecture for simple Lie groups of rank at least $2$.}

When applied to simple Lie groups of real rank at least $2$, Theorem \ref{thm:stationary-Stuck--Zimmer} reads as:

\begin{thm}\label{thm:stationary-simple}
Let $G$ be a centre-free higher rank simple Lie group. Let $\mu$ be a smooth probability measure on $G$. Let $\nu$ be an ergodic $\mu$-stationary random discrete subgroup of $G$. Then $\nu=\delta_{\{1\}}$, or there exists a lattice $\Gamma\subset G$ such that $\nu=\nu_\Gamma$. 
\end{thm}

This allows us to prove:

\begin{thm}\label{thm:rw-rank-1}
Let $G$ be a centre-free higher rank simple Lie group. Let $\gL\le G$ be a discrete subgroup of infinite covolume.  Then 
$$
\frac{1}{n}\sum_{i=0}^{n-1}\mu_G^{(i)}*\gd_{\gL}\to\gd_{\{1\}}.
$$
\end{thm}

\begin{proof}
Let $\nu$ be 
%an ergodic component of 
a stationary limit of $\frac{1}{n}\sum_{i=0}^{n-1}\mu_G^{(i)}*\gd_\gL$.
Then $\nu$ is supported on the closure of the conjugacy class of $\gL$ in $\sub(G)$.
Moreover, in view of Theorem \ref{thm:dSRS}, $\nu(\sub_d(G))=1$.
Since $G$ is simple of higher rank, lattices in $G$ are locally rigid. Thus the closure of the conjugacy class of $\gL$ in $\sub(G)$ contains no lattices. 
It follows from Theorem \ref{thm:stationary-simple} that $\nu=\gd_{\{1\}}$.
\end{proof}

Theorem \ref{thm:rw-rank-1} can be reformulated as follows:
Let $M=\gL\backslash G/K$ be a locally symmetric space of infinite volume, where $G$ is a simple Lie groups of rank at least $2$. Consider the $\mu_G$-random walk on $M$ starting at some point $x_0$.
Then for every $r>0$, the random walk will almost certainly eventually spend most of the time in the $r$-thick part. More precisely, let $\ti x_0\in G/K$ be a lift of $x_0$ and let $m_n\in \text{Prob}(M)$ be the pushforward of 
$\frac{1}{n}\sum_{i=0}^{n-1}\mu_G^{(i)}*\gd_{\ti x_0}$ via the covering map. Then we have:

\begin{cor}\label{cor:RW-on-M}
For every $r>0$ and $\gep>0$ there is $N$ such that, for $n\ge N$, 
$m_n(M_{\ge r})>1-\gep$. 
\qed
\end{cor}

As a straightforward consequence of Theorem \ref{thm:rw-rank-1}, we obtain the following equivalence: 

\begin{thm}\label{thm:simple-slim}
Let $G$ be a centre-free higher rank simple Lie group. Let $\gD\le G$ be a discrete subgroup. Then $\gD$ is confined if and only if $\gD$ is a lattice in $G$.
\end{thm}

\begin{proof}
Latices are obviously confined. On the other hand if $\gD\le G$ is of infinite covolume, it follows from Theorem \ref{thm:rw-rank-1} that $\{1\}$ is a conjugate limit of $\gD$. Thus $\gD$ is not confined.
%Let $\nu$ be 
%an ergodic component of 
%a stationary limit of $\frac{1}{n}\sum_{i=0}^{n-1}\mu_G^{(i)}*\gd_\gD$. Then $\nu$ is supported on the closure of the conjugacy class of $\gD$ in $\sub(G)$.
%Moreover, in view of Theorem \ref{thm:dSRS}, $\nu(\sub_d(G))=1$.
%Since $G$ is simple of higher rank, lattices in $G$ are locally rigid. Thus the closure of the conjugacy class of $\gD$ in $\sub(G)$ contains no lattices. 
%It follows from Theorem \ref{thm:stationary-simple} that $\nu=\gd_{\{1\}}$.
%it is enough to show that $\nu$ is not $\nu_\gC$ for some lattice $\gC$. 
%Now if $\nu=\nu_\gC$ then $\gC$ is a conjugate limit of $\gD$. However $\gC$ being a lattice in a higher rank simple group is locally rigid.  
%which implies that in that case $\gD$ contains a conjugate of $\gC$ and hence that $\gD$ is a lattice in contrast to our assumption. 
\end{proof}

%The proof 
%%of Theorem \ref{thm:simple-slim} 
%above shows that
%% if $G$ is a centre-free simple group of rank at least $2$ and 
%if $\gL\le G$ is a discrete subgroup of infinite covolume then 
%$$
%\frac{1}{n}\sum_{i=0}^{n-1}\mu_G^{(i)}*\gd_{\gL}\to\gd_{\{1\}}.
%$$
%weakly converges to $\gd_{\{1\}}$. 
%\begin{proof}
%Set
%$$
% \nu_n = \frac{1}{n}\sum_{i=0}^{n-1}\mu_G^{(i)}*\gd_{\gL}\in \text{Prob}(\sub_d(G)),
%$$
%and let $\nu_\infty$ be a stationary limit of $\nu_n$. By Theorem \ref{thm:dSRS} $\nu_\infty$ is a discrete stationary random subgroup. 
%The proof of Theorem \ref{thm:simple-slim} rules out the case that $\nu_\infty=\nu_\gC$ for a lattice $\gC\le G$. Therefore by Theorem \ref{thm:stationary-simple} $\nu_\infty=\gd_{\{1\}}$. Thus $\gd_{\{1\}}$ is the unique limit of $\nu_n$ and the statement of Corollary \ref{cor:RW-on-M} follows. 
%\end{proof}

%%%%%%%%%%%%%%
%%%%%%%%%%%%%%
%%%%%%%%%%%%%%

\section{Confined subgroups of semisimple groups}\label{sec:US}

Throughout this section we will assume that $G=G_1\times\cdots\times G_m$ is a connected centre-free semisimple Lie group without compact factors.
The $G_i$'s are the simple factors of $G$ and we suppose that $m\ge 2$.
Since the normal subgroup theorem requires higher rank and the Stuck--Zimmer theorem requires in addition property (T), the rank one factors play a special role. 

Recall that there are three infinite families and one exceptional simple real Lie group of rank one. Among them 
 $\text{Sp}(n,1),~n\ge 2$ and $F_{4(-20)}$ have Kazhdan's property (T) while the real and complex hyperbolic groups $\text{PO}(n,1)$ and $\text{PU}(n,1),~n\ge 2$ do not have property (T). 
 
 \begin{defn}
 We shall say that a semisimple group $G$ is of type (L) if every simple factor of $G$ is either $\text{PO}(n,1)$ or $\text{PU}(n,1),~n\ge 2$ (the letter (L) is after Lubachevsky.).
 \end{defn}

\begin{thm}\label{thm:g.s.s.-R-W}
Let $G$ be a connected centre-free semisimple Lie group without compact factors. Let $\gL$ be a discrete subgroup of $G$ and suppose that for every nontrivial normal subgroup $H\lhd G$, the intersection $\gL\cap H$ is not a lattice in $H$. Let $\nu$ be an ergodic component of a stationary limit of $\frac{1}{n}\sum_{i=0}^{n-1}\mu_G^{(i)}*\gd_{\gL}$. Then there are (possibly trivial) normal subgroups $H_1,H_2\lhd G$, where $H_1$ is a product of rank one factors, $H_2$ is of type (L) and $H_1\cap H_2=\{1\}$, such that $\nu$-almost surely:
\begin{enumerate}
\item the random group is contained in $H_1\times H_2$ and projects discretely to $H_1$ and to $H_2$,
\item the projection to $H_2$ is an IRS whose irreducible factors are higher rank and the intersection with each factor is thin (i.e. Zariski dense of infinite covolume in the factor),
\item the intersection with every simple factor of $H_1$ is Zariski dense, and
\item the intersection with every simple factor of $H_1$ with property (T) is thin.
\end{enumerate}
\end{thm}

\begin{rem}
It is a famous open problem whether the analog of the Stuck--Zimmer theorem holds for higher rank groups of type (L). A positive solution to this problem would imply that the group $H_2$ in Theorem \ref{thm:g.s.s.-R-W} is trivial. On the other hand the group $H_1$ may not be trivial even if $\gL$ is not confined. Using ideas from \cite{7s-2} one can construct a non-confined subgroup of a rank one group whose stationary limit is supported on confined subgroups. %In \S \ref{sec:example} we describe an example of a discrete subgroup of $\SL_2(\BR)\times\SL_2(\BR)$ with trivial intersections with the factors  whose stationary limit is a confined subgroup of one factor. 
\end{rem}

\begin{rem}
For certain specific cases we can say more about the stationary limit. For instance if $\gL$ does not contain large commuting subgroups then so does a random group with respect to a stationary limit. Such data eliminates some of the possibilities that the general case allows. In the special case that $\gL$ is an infinite index subgroup of an irreducible higher rank lattice and $G$ has property (T), it is easy to deduce that the stationary limit must be $\gd_{\{1\}}$. 
%In that special case, the same holds for the random walk associated with an atomic measure supported on a generating set for an ambient lattice. 
\end{rem}

\begin{lem}\label{lem:Wang}
Let $H\lhd G$ be a normal subgroup of $G$ and let $\gC\le H$ be a lattice. Suppose that either $\rank(H)\ge 2$ and $\gC$ is irreducible or that $H$ has Kazhdan's property (T). Then $\gC$ is locally rigid in $G$, 
that is, the $G$-conjugacy class of the inclusion $\iota:\gC\to G$ is open in $\text{Hom}(\gC,G)$.
\end{lem}

\begin{proof}
The case where $\rank(H)\ge 2$ follows from \cite[page 332]{Ma} and \cite[Theorem 6.7]{Rag}. Therefore we may suppose that $H$ has property (T). Let $f:\gC\to G$ be a deformation and consider the unitary representation of $\gC$ on $L^2(G/H)$ where $\gC$ acts via $f$.
If $f$ is a sufficiently small deformation this representation has a $(\gS,\gep)$-invariant unit vector where $\gS$ is a finite generating set for $\gC$ and $\gep=\gep(\gS)>0$ is a Kazhdan constant for $\gS$ given by property (T) of $\gC$. It follows that this representation has a nontrivial invariant vector which implies that the projection of $f(\gC)$ lies in a compact subgroup of $G/H$. By further conjugating by some element close to the identity we may assume that it lies in a fixed compact subgroup $K\le G/H$. By \cite[Theorem 2.6]{Wang} $\gC$ has only finitely many conjugacy classes of representations into $K$. Thus we deduce that eventually $f(\gC)\le H$. 
The result now follows from local rigidity of $\gC$ in $H$.
\end{proof}

\begin{cor}\label{cor:Wang}
Let $H\lhd G$ be as in Lemma \ref{lem:Wang}. 
Let $\Lambda$ be a discrete subgroup of $G$. 
Suppose that $\Lambda$ admits a discrete conjugate limit which intersects $H$ by an irreducible lattice. Then $\Lambda$ intersects $H$ in a lattice. 
\end{cor}

\begin{proof}
Let $\gD$ be a discrete conjugate limit of $\Lambda$ and suppose that $\gC=\gD\cap H$ is a lattice in $H$. Fixing a finite presentation $\langle\gS:R\rangle$ for $\gC$ we see that if $\Lambda^g$ is sufficiently close to $\gD$ then the map sending each $\gs\in\gS$ to its nearest element in $\Lambda^g$ extends to a homomorphism $f_g:\gC\to\Lambda^g$. Thus the corollary follows from Lemma \ref{lem:Wang}. 
\end{proof}

\begin{proof}[Proof of Theorem \ref{thm:g.s.s.-R-W}]
%Let $\nu_\infty$ be a stationary limit of $\nu_n=\frac{1}{n}\sum_{i=0}^{n-1}\mu^{(i)}*\gd_\Lambda$. 
By Theorem \ref{thm:dSRS}, $\nu$ is a discrete stationary random subgroup. Consider the decomposition of $G$ given by Theorem \ref{thm:stationary-decomposition}, $G=G_\mathcal{I}\times G_\mathcal{H}\times G_\mathcal{T}$ and set $H_1=G_\mathcal{H}, H_2=G_\mathcal{I}$.
By Theorem \ref{thm:stationary-decomposition}, $H_1$ is a product of rank one factors.
Suppose by way of contradiction that $H_2$ is not of type (L). Then by Proposition \ref{prop:IRS-lattices}, Remark \ref{rem:commensurable} and Remark \ref{rem:Hartman} a random subgroup intersects a higher rank semisimple factor of $G_\mathcal{I}$ by an irreducible lattice. This is also the case if Item (2) of the theorem is not satisfied. Similarly, if Item (4) of the theorem is not satisfied then a random subgroup intersects by a lattice some rank one simple factor of $G$ which has property (T). In all three cases we deduce from Corollary \ref{cor:Wang} that $\gL$ itself intersects the given factor in a lattice. This is a contradiction to the assumption. 
\end{proof}

%We remark that the same result holds also in the luck of property (T) when $H$ is higher rank and $\gD\cap H$ is irreducible (see \cite{Shalom} for a similar statement in that context).

We deduce the following classification of confined subgroups when all the factors of $G$ are higher rank:

\begin{thm}
Let $G=G_1\times\cdots\times G_n$ be a connected centre-free semisimple Lie group such that each $G_i$ is simple of rank at least $2$.
Let $\Lambda$ be a discrete subgroup. Then $\Lambda$ is confined if and only if there is a nontrivial semisimple factor $H\lhd G$ such that $\Lambda\cap H$ is a lattice in $H$. If $\Lambda$ is not confined then 
 $\frac{1}{n}\sum_{i=0}^{n-1}\mu_G^{(i)}*\gd_{\gL}\to\gd_{\{1\}}$.
\end{thm}

\begin{defn}
Let $G$ be a locally compact group. A subset $D\subset G\times G$ will be called a {\it density tester} if it satisfies the following property: A subgroup $\Gamma\le G$ is dense if and only if $(\Gamma\times\Gamma)\cap D\ne \emptyset$.
\end{defn}

Note that any connected semisimple Lie group $G$ admits a compact density tester. Indeed, we can take two compact sets with non-empty interior $U,V$ lying in a Zassenhaus neighbourhood of $G$ where the logarithm is well defined such that for any $u\in U,v\in V$, we have that $\log(u)$ and $\log(v)$ generate the Lie algebra of $G$. Then it follows that for any such pair the group $\langle u,v\rangle$ is dense in $G$ (see \cite{dense} or \cite{Zuk} for details). Thus $D=U\times V$ is a density tester.

\begin{defn}
Let $G=G_1\times\cdots\times G_n$ be a connected semisimple Lie group with simple factors $G_i$. Let $\Gamma\le G$ be a discrete subgroup. We will say that $\Gamma$ is {\it uniformly irreducible} if there are:
\begin{itemize}
\item a compact density tester $D_i$ for $G/G_i$ for $i=1,\ldots,n$  
\item a compact set $B\subset G$
\end{itemize}
such that for every $g\in G$ and $i$, the projection of the finite set $\gC^g\cap B$ to $G/G_i$ admits two points $\gc_i^1,\gc_i^2$ such that $(\gc_i^1,\gc_i^2)\in D_i$.
\end{defn}

\begin{rem}
Obviously a uniformly irreducible subgroup is confined.
It is easy to check that every irreducible lattice in $G$ is uniformly irreducible. It is also not hard to check that a nontrivial normal subgroup of a uniform irreducible lattice is uniformly irreducible.
\end{rem}

\begin{thm}\label{thm:uniform-irreducible}
Let $G=G_1\times\cdots\times G_n,~n\ge 2$ be a connected centre-free semisimple Lie group without compact factors and suppose that at least one of the simple factors of $G$ has Kazhdan's property (T). Let $\Lambda$ be a discrete subgroup of $G$. Then $\Lambda$ is uniformly irreducible if and only if $\Lambda$ is an irreducible lattice in $G$.
\end{thm}

\begin{proof}
Let $\Lambda\le G$ be a uniformly irreducible discrete subgroup. Note that every conjugate limit of a uniformly irreducible group is uniformly irreducible. 
In view of that, it follows from 
Theorem \ref{thm:dSRS}, Proposition \ref{prop:IRS-lattices} and Remark \ref{rem:Hartman} that every ergodic component of a stationary limit of  $\frac{1}{n}\sum_{i=0}^{n-1}\mu_G^{(i)}*\gd_{\gL}$, must be $\nu_\gC$ for some irreducible lattice in $G$. In view of Corollary \ref{cor:Wang} this implies that $\Lambda$ is an irreducible lattice in $G$.
\end{proof}

Let $G$ be as above and let $G_{i}$ be a simple factor. We shall say that a discrete subgroup $\Lambda$ projects uniformly densely to $G_{i}$ if there is a compact set $B\subset G$ and a compact density tester $D$ for $G_{i}$ such that for every $g\in G$ the projection of $\Lambda^g\cap B$ to $G_{i}$ contains two points $\ga,\gb$ such that $(\ga,\gb)\in D$. The following is a variant of Theorem \ref{thm:uniform-irreducible}. We leave the proof as an exercise.

\begin{prop}
Let $G=G_1\times\cdots\times G_n,~n\ge 2$ be a connected centre-free semisimple Lie group without compact factors and with Kazhdan's property (T). Let $\Lambda$ be a discrete confined subgroup of $G$. Suppose in addition that $\Lambda$ projects uniformly densely to every rank one factor of $G$. Then there is a nontrivial semisimple factor $H\lhd G$ which contains all the rank one factors of $G$ such that $\Lambda\cap H$ is a lattice in $H$.
\end{prop}

%We shall now address the general case where $G$ is a higher rank semisimple group with property (T) and $\gC$ is a discrete subgroup of $G$, without assuming any irreducibility assumption on $\gC$.
%
%
%\begin{thm}\label{thm:no-discrete-projections}
%Let $G=G_1\times\cdots\times G_n,~n\ge 2$ be a connected centre-free semisimple Lie group without compact factors and with Kazhdan's property (T). Let $\Lambda$ be a discrete subgroup of $G$. Suppose that $\Lambda$ is not confined and does not admit a discrete conjugate limit which intersects some rank one factor of $G$ in a Zariski dense subgroup.
%Then 
%$$
% \lim_{n\to\infty}\frac{1}{n}\sum_{i=1}^n\mu^{(i)}*\gd_\Lambda=\gd_{\langle 1\rangle}.
%$$
%\end{thm}
%
%\begin{proof}
%In view of Theorem \ref{thm:stationary-Stuck--Zimmer} and the assumption concerning the rank-one factors, we deduce that the stationary limit $\nu_\infty$ is a convex combination of measures of the form $\nu_\gC$ where $\gC$ is a lattice in a semisimple factor $H\lhd G$. However since $\Lambda$ is not confined, it follows from Lemma \ref{lem:Wang} that non-trivial $H$ cannot occur. Thus $\nu_\infty=\gd_{\{1\}}$. 
%\end{proof}

We shall now prove the following alternative:

\begin{thm}\label{thm:u.s.-alternative}
Let $G=G_1\times\cdots\times G_n,~n\ge 2$ be a connected centre-free semisimple Lie group without compact factors and with Kazhdan's property (T).
%Let $G$ be as in Theorem \ref{thm:no-discrete-projections}.
Let $\gL\le G$ be a discrete confined group. Then exactly one of the following holds:
\begin{enumerate}
\item[(i)]
There is a nontrivial proper semisimple factor $H\lhd G$ such that $\gL\cap H$ is a lattice in $H$.
\item[(ii)]
There is a proper nontrivial semisimple factor $H\lhd G$ which is a product of rank one simple factors, $H=\prod_{i=j}^kG_{i_j}$, and a conjugate limit $\gD$ of $\gL$, such that $\gD$ is a discrete confined subgroup of $H$.
Furthermore $\gD$ intersects every simple factor of $H$ in a confined thin subgroup.
\end{enumerate}
\end{thm}

%\begin{rem}
%If $\gC$ intersects trivially each factor of $G$ and is confined of type $(2)$ of Theorem \ref{thm:u.s.-alternative} then $H$ must be simple.
%\end{rem}

We shall rely on the following lemma which is also of independence interest. 

\begin{lem}\label{lem:u.s.-Z-dense}
A confined discrete subgroup of a simple Lie group is Zariski dense.
\end{lem}

\begin{proof}
Let $G$ be a simple group and $\gL\le G$ a discrete subgroup. If $\text{rank}(G)\ge 2$ then by Theorem \ref{thm:simple-slim}, $\gL$ is confined if and only if it is a lattice, hence the lemma follows from the Borel density theorem. Suppose therefore that $\text{rank}(G)=1$ and
suppose that $\gL$ is not Zariski dense. Then $\gL$ is contained in a maximal proper algebraic subgroup which is either reductive or parabolic (see \cite{BT}). In the latter case pick a Langlands decomposition of the parabolic and a central element in the Levi subgroup which acts by expansion on the unipotent radical of the parabolic. Then by applying iterative conjugations by this element all the elements in $\gL$ which do not belong to the Levi are taken to infinity and we obtain a conjugate limit which is contained in the Levi subgroup. Thus it is enough to deal with the case that $\gL\le H$ for some reductive subgroup $H\le G$. Let $X=G/K$ be the symmetric space of $G$. By a theorem of Mostow \cite{Mostow} there is a totally geodesic subspace $Y\subset X$ which is $H$-invariant and is a model for the symmetric space of $H$ ($Y$ is a single point in case $H$ is compact). Let $c:[0,\infty)\to X$ be any geodesic in $X$ which starts at a point $y\in Y$ such that $\dot{c}(0)$ is orthogonal to $Y$. Since $\text{rank}(G)=1$, $X$ is negatively curved and hence the displacement of any $h\in H$ goes to infinity along $c$. It follows that the injectivity radius of $\Gamma\backslash X$ goes to infinity along $c(t)$.
\end{proof}

%\begin{lem}\label{cor:disc-trivial}
%Let $H=H_1\times H_2$ be a semisimple Lie group and let $\gC\le H$ be a discrete group which projects discretely to $H_1$ and intersect trivially $H_2$. Then $\gC$ admits a conjugate limit which is discrete and contained in $H_1$.
%\end{lem}
%
%
%\begin{proof}
%%Let $\gC\le H$ be a discrete group which projects discretely to $H_1$ and intersect trivially $H_2$.
%Consider the $\mu_{H_2}$-stationary limit
%$$
% \lim_{n\to\infty}\frac{1}{n}\sum_{i=0}^{n-1}\mu_{H_2}^{(n)}*\gd_{\gC}.
%$$ 
%This measure is supported on conjugate limits of $\gC$ whose projection to $H_1$ is contained in the projection of $\gC$ to $H_1$. 
%The result follows from Corollary \ref{cor:nu_2-stationary}.
%
%\end{proof}

\begin{lem}\label{lem:u.s.-intersections}
Let $G=G_1\times\cdots\times G_n$ be a connected semisimple Lie group and $\gL\le H$ a discrete subgroup that projects discretely to the factors. Then $\gL$ admits a conjugate limit $\gD$ which is discrete and satisfies the following properties:
\begin{itemize}
\item
The projection of $\gD$ to every factor is contained in the projection of $\gL$ to the same factor. 
\item
The projection of $\gD$ to a factor of $G$ is trivial unless the intersection of $\gL$ with that factor is confined.
\end{itemize}
\end{lem}

\begin{proof}
Let $\gL_0=\gL$. For each $i=1,\ldots,n$ at its turn, if $\gL_{i-1}\cap G_i$ is not confined then we replace $\gL_{i-1}$ with a $G_i$-conjugate limit $\gL_i$ such that $\gL_i\cap G_i=\{1\}$. If $\gL_{i-1}\cap G_i$ is confined we let $\gL_i=\gL_{i-1}$. Then $\gD=\gL_n$ satisfies the desired requirements.
\end{proof}

\begin{proof}[Proof of Theorem \ref{thm:u.s.-alternative}]
Let $\nu$ be an ergodic component of a weak limit of  $\frac{1}{n}\sum_{i=0}^{n-1}\mu_G^{(i)}*\gd_{\gL}$, and consider the decomposition of $G$ according to $\nu$ given by Theorem \ref{thm:stationary-decomposition}. If $G_\mathcal{I}$ is nontrivial 
then by Proposition \ref{prop:IRS-lattices} and Remark \ref{rem:commensurable} the restriction of $\nu$ to $G_\mathcal{I}$ is of the form $\nu_\gC$ for some lattice $\gC\le G_\mathcal{I}$. In view of Corollary \ref{cor:Wang}, $\gL$ intersects $G_\mathcal{I}$ by a lattice.
%Similarly if $\nu$ projects to a lattice IRS in one of the rank one factors $G_i$ of $G_\mathcal{H}$ then by the same reasoning we conclude that $\gL\cap G_i$ is a lattice in $G_i$.

Consider now the case that $G_\mathcal{I}$ is trivial.
Let $\gD$ be a generic subgroup in the support of $\nu$. Then $G$ decomposes as $G=G_\mathcal{H}\times G_\mathcal{T}$ such that $\gD\le G_\mathcal{H}$, all the factors of $G_\mathcal{H}$ are of rank one, the projection of $\gD$ to every factor of $G_\mathcal{H}$ is discrete and the intersection of $\gD$ with every factor of $G_\mathcal{H}$ is Zariski dense. If the intersection of $\gD$ with one of the simple factors of $G_\mathcal{H}$ is a lattice then by Corollary \ref{cor:Wang} also $\gL$ intersects this factor by a lattice. 
Finally we apply Lemma \ref{lem:u.s.-intersections} to the confined group $\gD$ and the theorem follows. 
\end{proof}

\begin{rem}
The assumption that $G$ in Theorem \ref{thm:u.s.-alternative} has property (T) is made because we do not know if the analog of the
Stuck--Zimmer theorem holds for higher rank groups of type (L). If that analog holds then one can remove the property (T) assumption without changing the statement. 
Without knowing the answer to this question one can still drop the property (T) assumption by
allowing in Case (ii) of the theorem that $H$ admits beside the rank one factors also higher rank semisimple factors $H_i$ of type (L), and $\gD\cap H_i$ is thin in $H_i$ and confined and projects densely to the simple factors $H_i$. We do not know however if such `irreducible' thin confined discrete subgroups exist. 
\end{rem}


\begin{thebibliography}{999999}

\bibitem[7s17]{7s}
M. Abert, N. Bergeron, I. Biringer, T. Gelander, N. Nikolov, J. Raimbault, I.  Samet,  On the growth of $ L^ 2$-invariants for sequences of lattices in Lie groups, Annals of Math. 185 (2017), 711--790.

\bibitem[7s20]{7s-2}
M. Abert, N. Bergeron, I. Biringer, T. Gelander, N. Nikolov, J. Raimbault, I.  Samet,
On the growth of L2-invariants of locally symmetric spaces, II: exotic invariant random subgroups in rank one. Int. Math. Res. Not. IMRN 2020, no. 9, 2588--2625.

\bibitem[BSh06]{Bader-Shalom}
U. Bader, Y. Shalom, Factor and normal subgroup theorems for lattices in products of groups.
Invent. Math. 163 (2006), no. 2, 415--454.

\bibitem[BBHP20]{BBHP}
 U. Bader, R. Boutonnet, C. Houdayer, J. Peterson, Charmenability of arithmetic groups of product type, arXiv:2009.09952.
 
\bibitem[BH20]{BH} R. Boutonnet, C. Houdayer, Stationary characters on lattices of semisimple Lie groups, arXiv:1908.07812.
 
\bibitem[BG03]{dense}
E. Breuillard, T. Gelander, On dense free subgroups of Lie groups. J. Algebra 261 (2003), no. 2, 448--467.

\bibitem[Bor]{Bor} A. Borel. Linear algebraic groups. Vol. 126. Graduate Texts in Mathematics Springer, 1991.

\bibitem[BT71]{BT} A. Borel and J. Tits. \'El\'ements unipotents et sous-groupes paraboliques de groupes
r\'eductifs. I. Invent. Math., 12:95--104, 1971.

\bibitem[C20]{Creutz} D. Creutz, Stabilizers of Stationary Actions of Lattices in Semisimple Groups, arXiv:2010.13987.

\bibitem[CM93]{CM} D.H. Collingwood, David, W.M. McGovern, Nilpotent orbits in semisimple Lie algebras, Van Nostrand Reinhold Co., New York, 1993. xiv+186 pp. ISBN: 0-534--18834--6.


\bibitem[Fur63]{Fur63} H. Furstenberg, A Poisson formula for semi-simple Lie groups. Ann. Math. (2)
77, 335--386, 1963.

\bibitem[LG19]{LG} I. Gekhtman, A. Levit, Critical exponents of invariant random subgroups in negative curvature. Geom. Funct. Anal. 29, (2019) 411–-439.

\bibitem[G18]{KM-IRS} T. Gelander, Kazhdan-Margulis theorem for invariant random subgroups. Adv. Math. 327 (2018), 47--51.

\bibitem[GL18]{GL} T. Gelander, A. Levit, Invariant random subgroups over non-Archimedean local fields. Math. Ann. 372 (2018), no. 3-4, 1503--1544.
%\bibitem[G15]{G} T. Gelander, Lecture notes on Invariant Random Subgroups and Lattices in rank one and higher rank, arXiv:1503.08402.
%
\bibitem[GLM]{GLM} T. Gelander, A. Levit, G.A. Margulis, Effective discreteness radius of stabilisers for stationary actions, preprint.

\bibitem[GZ02]{Zuk}
T. Gelander, A. \.Zuk, Dependence of Kazhdan constants on generating subsets. Israel J. Math. 129 (2002), 93--98.

\bibitem[GW15]{GW}
E. Glasner, B. Weiss, Uniformly recurrent subgroups. Recent trends in ergodic theory and dynamical systems, 63--75, Contemp. Math., 631, Amer. Math. Soc., Providence, RI, 2015.

\bibitem[GV88]{GV}
R. Gangolli, V.S. Varadarajan. Harmonic analysis of spherical functions on real reductive groups. Vol. 101. Springer Science \& Business Media, 2012.

\bibitem[HT16]{HT} Y. Hartman, O. Tamuz, Stabilizer rigidity in irreducible group actions. Israel J. Math. 216 (2016), no. 2, 679--705.
 
\bibitem[HM79]{HM79} R. Howe, C. C. Moore. "Asymptotic properties of unitary representations." J. Func-tional Anal. 32 (1979), 72--96.

\bibitem[Kak51]{Kak51} S. Kakutani, Random Ergodic Theorems and Markoff Processes with a Stable Distribution, Second Berkeley Symposium on Mathematical Statistics and Probability, 247--261, 1951.

\bibitem[M90]{Ma} G.A. Margulis, Discrete Subgroups of Semisimple Lie Groups, Springer-Verlag, 1990.

\bibitem[M78]{Ma-FT} G.A. Margulis, Factor groups of discrete subgroups and measure theory. (Russian)
Funktsional. Anal. i Prilozhen. 12 (1978), no. 4, 64--76.

\bibitem[Kn02]{Knapp}  A. W. Knapp, Lie groups beyond an introduction. Second edition. Progress in Mathematics, 140. Birkh\"auser Boston, Inc., Boston, MA, 2002.\

\bibitem[NZ99]{NZ99} A. Nevo, Amos, R.J. Zimmer, Homogenous projective factors for actions of semi-simple Lie groups. Invent. Math. 138 (1999), no. 2, 229--252.

\bibitem[NZ02]{NZ} A. Nevo, Amos, R.J. Zimmer, A structure theorem for actions of semisimple Lie groups. Ann. of Math. (2) 156 (2002), no. 2, 565--594.

\bibitem[NZ02b]{NZ-IFT} A. Nevo, Amos, R.J. Zimmer, A generalization of the intermediate factors theorem. J. Anal. Math. 86 (2002), 93--104. 

\bibitem[SZ94]{SZ} G. Stuck, R. J. Zimmer. "Stabilizers for Ergodic Actions of Higher Rank Semisimple Groups." Annals of Math. 139 (1994), 723--47.

\bibitem[EM02]{EM} A. Eskin, G. Margulis. "Recurrence properties of random walks on finite volume homogeneous manifolds." Random walks and geometry (2002), 431--444.

\bibitem[Mo55]{Mostow} G. D. Mostow, Some new decomposition theorems for semi-simple groups, Mem. Amer. Math. Soc. 14 (1955), 31--54.

\bibitem[Rag89]{Rag} H. C. Raghunathan, Discrete subgroups of algebraic groups over local fields of positive characteristics, Indian Academy of Science. Proceedings. Mathematical Sciences 99 (1989), 127--146.

%\bibitem[Sh]{Shalom} Y. Shalom, Rigidity of commensurators and irreducible lattices. Invent. Math. 141 (2000), 1--54.
\bibitem[Wa75]{Wang} P.S. Wang, On isolated points in the dual spaces of locally compact groups. Math. Ann. 218 (1975),
19--34.

\end{thebibliography}
\end{document}